\documentclass[11pt]{amsart}
 \usepackage{amssymb}
 \usepackage[all]{xy}
\usepackage{colordvi}
\usepackage[dvips]{graphicx}
\usepackage{pdfsync}
\usepackage{url}
\usepackage[ps2pdf]{hyperref}

\usepackage{algorithmic}
\usepackage{algorithm}
\usepackage{array}
\usepackage[latin1]{inputenc}
\usepackage{vmargin}
\setpapersize{USletter}
\setmargrb{1in}{1in}{1in}{1in} 

\newtheorem{thm}{Theorem}[section]
\newtheorem{cor}[thm]{Corollary}
\newtheorem{conj}[thm]{Conjecture}

\newtheorem{prob}[thm]{Problem}
\theoremstyle{definition}
\newtheorem{ex}[thm]{Example}
\newtheorem{lem}[thm]{Lemma}
\newtheorem{prop}[thm]{Proposition}
\theoremstyle{definition}
\newtheorem{defn}[thm]{Definition}
\newtheorem{rem}[thm]{Remark}
\numberwithin{equation}{section}

\def \N { {\mathbb N} }

\def \Z { {\mathbb Z} }
\def \P { {\mathbb P} }
\def \a { {a} }  
\def \b { {b} }
\def \cc {{c} }

\newcommand{\K}{\mathbb{K}}
\newcommand{\symm}{\mathfrak{S}}
\newcommand{\matrA}{\mathcal{A}}
\newcommand{\matrB}{\mathcal{B}}
\newcommand{\rtilde}{\widetilde{R}}
\newcommand{\atilde}{\widetilde{\a}}
\newcommand{\btilde}{\widetilde{\b}}
\newcommand{\gtilde}{\widetilde{\cc}}

\newcommand{\<}{\langle}
\renewcommand{\>}{\rangle}
\DeclareMathOperator{\diag}{diag}
\DeclareMathOperator{\spanZ}{Span_\Z}
\DeclareMathOperator{\spanN}{Span_\N}

 \title[Finiteness for chains of Laurent lattice ideals]{Finiteness theorems and algorithms for \\ permutation invariant chains of Laurent lattice ideals}

\author[Hillar]{Christopher J. Hillar}
\address{Christopher J. Hillar\\
 	 The Mathematical Sciences Research Institute\\
         17 Gauss Way\\
         Berkeley, CA 94720-5070\\
         USA}
\email{chillar@msri.org}
\urladdr{http://www.msri.org/people/members/chillar}
\author[Mart\'in del Campo]{Abraham Mart\'in del Campo}
\address{Abraham Mart\'in del Campo\\
	  Department of Mathematics\\
         Texas A\&M University\\
         College Station, TX 77843-3368\\
         USA}
\email{asanchez@math.tamu.edu}
\urladdr{http://www.math.tamu.edu/~asanchez}

\thanks{The first author was partially supported by an NSA Young Investigators Grant and an NSF All-Institutes Postdoctoral Fellowship administered by the Mathematical Sciences Research Institute through its core grant DMS-0441170.}
\thanks{The second author was
  supported in part by NSF grant 
  DMS-915211.}

\keywords{Lattice ideal, toric ideal, invariant ideals, chain stabilization, symmetric group, finiteness, permutation module, nice orderings.}

\PII{}
\copyrightinfo{}{}
\begin{document}

\begin{abstract}
We study chains of lattice ideals that are invariant under a symmetric group action.  In our setting, the ambient rings for these ideals are polynomial rings which are increasing in (Krull) dimension.  Thus, these chains will fail to stabilize in the traditional commutative algebra sense.  However, we prove a theorem which says that ``up to the action of the group", these chains locally stabilize.  We also give an algorithm, which we have implemented in software, for explicitly constructing these stabilization generators for a family of Laurent toric ideals involved in applications to algebraic statistics.  We close with several open problems and conjectures arising from our theoretical and computational investigations. 

\end{abstract}

\maketitle

\section{Introduction}\label{intro}

In commutative algebra, finiteness plays a significant role both theoretically and computationally.  An important example is Hilbert's basis theorem, which states that any ideal $I \subseteq R$ in a polynomial ring $R = \mathbb C[x_1,\ldots,x_n]$ over the complex numbers $\mathbb C$ (or more generally, over any field $\mathbb K$) has a finite set of generators $G = \{g_1,\ldots,g_m\}$:  \[I =  \langle G\rangle_R := g_1R + \cdots + g_m R.\]  In other words, $\mathbb C[x_1,\ldots,x_n]$ is a \textit{Noetherian} ring.  Equivalently, any ascending chain of ideals $I_1 \subseteq I_2 \subseteq \cdots$ in $\mathbb C[x_1,\ldots,x_n]$ stabilizes (i.e., there exists an $N$ such that $I_N = I_{N+1} = \cdots$). This result has many applications in the algebraic theory of polynomial rings (e.g. the existence of finite resolutions~\cite[p. 340]{Eisenbud95}), but it is also a fundamental fact underlying computational algebraic geometry (e.g. termination of Buchberger's algorithm in the theory of Gr\"obner bases~\cite[p. 90]{CLO07}).  

In many contexts, however, finiteness is observed even though Hilbert's basis theorem does not directly apply.  A motivating example is the (non-Noetherian) ring $R = \mathbb C[x_1,x_2,\ldots]$ of polynomials in an \textit{infinite} number of indeterminates $X = \{x_1,x_2,\ldots\}$, equipped with a permutation action on indices.  More precisely, the \textit{symmetric group} $\symm_{\P}$ of all permutations of the positive integers $\P :=\{1,2,\ldots\}$ acts naturally on $R$ via:
\begin{equation}\label{group_poly_action}
\sigma f(x_{s_1},\ldots,x_{s_{\ell}}) := f(x_{\sigma(s_1)},\ldots, x_{\sigma(s_{\ell})}), \ \ \sigma \in \symm_{\P}, \ f \in R.
\end{equation}
Although many ideals in the ring $R$ are not finitely generated, an important subclass still admit finite presentations.  Call an ideal $I$  \textit{permutation-invariant} if it is fixed under the action of $\symm_{\P}$:
\[\symm_{\P} I:=\{ \sigma f : \sigma \in\symm_{\P}, f\in I\} = I.\]
It is known that for every such permutation-invariant $I \subseteq R$, there is a finite set of generators $G = \{g_1,\ldots,g_m\} \subset I$ giving it a presentation of the form: \[ I = \langle \symm_{\P} G \rangle_{R}.\]  
As a simple example, the ideal $M \subset \mathbb C[x_1,x_2,\ldots]$ of polynomials without constant term has the finite presentation $M = \langle \symm_{\P} x_1 \rangle_{\mathbb C[x_1,x_2,\ldots]}$ even though it is not finitely generated. 

The above finiteness property for the ring $\mathbb C[x_1,x_2,\ldots]$ was first discovered by Cohen in the context of group theory \cite{Cohen67} (see also \cite{Cohen77} for algorithmic aspects), but seems to have gone unnoticed in the commutative algebra community until its independent rediscovery recently in \cite{AH07}.  Generalizations and extensions of this result have since been applied to unify several finiteness results in algebraic statistics \cite{HS} as well as help prove open conjectures in that field (notably, the independent set conjecture \cite{Hosten2007, HS}, finiteneness for the $k$-factor model \cite{Draisma2010}, and, more recently, that bounded-rank tensors are defined in bounded degree \cite{DraismaKuttler11}). 

In this paper,  we derive new finiteness properties for certain classes of polynomial ideals that are invariant under a symmetric group  action.  Motivated by an algebraic question of Dress and Sturmfels in chemistry  \cite[Section 5]{AH07}, we prove that invariant chains of lattice ideals stabilize up to monomial localization (see Theorem \ref{laurentlatticethm} below). This general result gives evidence for Conjecture 5.10 in \cite{AH07} (stated as Conjecture \ref{conjwmonomial} below).  Moreover, for the specific chains studied there (in \cite[Section 5.1]{AH07}), we present an algorithm for explicitly constructing these generators (see Theorem \ref{ouralgthm} and Algorithm \ref{algorithmourthm} below).  Our results also have potential implications for algebraic statistics. To prepare for the precise statements, however, we need to introduce some notation.  

Given a set $S$, let $\symm_S$ denote the group of permutations of $S$.  We shall focus our attention primarily on the sets $S= [n] :=\{1,2,\ldots, n\}$ and $S= \P :=\{1,2,\ldots\}$, the set of positive integers.   In these cases, we write $\symm_n$ and $\symm_\P$,  respectively, for the symmetric groups.\footnote{We embed $\symm_n$ into $\symm_m$ for $n \leq m$ in the natural way.}  Given a positive integer $k\geq 1$, let $[S]^k$ be the set of all ordered $k$-tuples $u=(u_1,\ldots,u_k)$, and let $\<S\>^k$ be the subset of those with pairwise distinct $u_1,\ldots, u_k$. When $S=[n]$, we write $[n]^k$ and $\<n\>^k$ for $[S]^k$ and $\<S\>^k$, respectively.  

The symmetric group $\symm_S$ acts on $[S]^k$ naturally via
\begin{equation}\label{SymmPk}
\sigma(u_1,\ldots,u_k):=(\sigma(u_1),\ldots,\sigma(u_k)), \quad \text{for } \sigma\in\symm_S;
\end{equation}
and this action restricts to an action on $\<S\>^k$.

Write $X_S :=  \{x_s: s \in S \}$ for the set of indeterminates indexed by a set $S$, and let $\K[X_S]$ denote the polynomial ring with coefficients in a field $\K$ (e.g.,  $\mathbb C$ or $\mathbb R$) and indeterminates $X_S$.  The action of any group $\symm$ on $S$ induces an action on $X_S$, which we extend to an action on $\K[X_S]$ as in (\ref{group_poly_action}).

We are interested here in the highly structured \textit{$\symm$-invariant} ideals of $\K[X_S]$
(simply called \textit{invariant} ideals below if the group $\symm$ is understood);
these are ideals $I \subseteq \K[X_S]$ for which $\symm I =  I$.\footnote{In the language of \cite{AH09}, invariant ideals are also the $\K[X_S] \ast \symm$-submodules of $\K[X_S]$, where $\K[X_S] \ast \symm$ is the skew group ring associated to $\K[X_S]$ and $\symm$.} Guised in various forms, invariant ideals of polynomial rings arise naturally in many contexts.  For instance, they appear in applications of polynomial algebra to chemistry~\cite{RSU67, AH07, Draisma2010}, finiteness of statistical models in algebraic statistics and toric algebra~\cite{Santos2003, SturmfelsSullivant05, Kuo06, DSS07, Hosten2007, AH07, DLS09, Brouwer-Draisma11, AHT10, TakemuraHara10b, TakemuraHara10a, Draisma2010, Snowden10, DraismaKuttler11, HMCY2011, HS}, and the algebra of tensor rank \cite{DraismaKuttler11}. 

Given an ideal $I \subseteq R$ of a polynomial ring $R = \K[X_S]$, let $I^{\pm}$ denote the localization $I \hookrightarrow I^{\pm}$ of $I$ with respect to the multiplicative set of monomials of $R$ (including the monomial $1$). In particular, $R^{\pm}$ is the ring of \textit{Laurent polynomials} in the indeterminates of $R$, and any ideal $I \subseteq R$ lifts to an ideal $I^{\pm} \subseteq R^{\pm}$, which we call a \textit{Laurent ideal}.  In simple terms, the ideal $I^\pm$ consists of elements of the form $gh^{-1}$ where $g\in I$ and $h$ is a monomial of $R$ (see e.g. ~\cite{Eisenbud95}). An action of a group $\symm$ of automorphisms that permute the indeterminates of $R$ extends naturally to an action on $R^{\pm}$: for $\sigma\in \symm$ and $g h^{-1}\in R^\pm$, we can define $\sigma(g h^{-1}):= \sigma(g) \sigma(h)^{-1}\in R^\pm$.  In this way, any $\symm$-invariant ideal $I$ lifts to an $\symm$-invariant ideal $I^{\pm} \subseteq R^{\pm}$. 
As above, for a subset $G\subseteq R$, we let $\<G\>_R$ denote the ideal generated by $G$ over $R$.

In this paper, we work with localized (Laurent) ideals because they allow us to prove very general finiteness theorems in cases where no other known techniques are able to produce such results.

Fix a positive integer $k$.  In what follows, we are primarily concerned with the polynomial rings (and their localizations):
\begin{equation}\label{Rndefs}
\mathcal{R}_n := \K[ X_{ [ n ]^k }], \ \ \mathcal{R}_\P := \K[ X_{[\P]^k}] = \bigcup_{n \in \P} \mathcal{R}_n; \ \  R_n := \K[ X_{ \< n \>^k }], \ \ R_\P = \bigcup_{n \in \P} R_n;
\end{equation}
and $T_n := \K[t_1,\ldots,t_n]$.  Since the set $[ n ]^k$ sits naturally inside $[ m ]^k$ for $n \leq m$, we have an embedding of rings $\mathcal{R}_n\subseteq \mathcal{R}_m$; similarly,  $R_n \subseteq R_m$. Our main objects of interest will be \emph{ascending chains} $I_\circ$ of ideals $I_n\subseteq \mathcal{R}_n$ (simply called \textit{chains} below):
\begin{equation}\label{chaindef}
I_\circ := I_1\subseteq I_{2}\subseteq \cdots.
\end{equation}
In general, a chain of ideals (\ref{chaindef}) will not stabilize in the sense of Hilbert's basis theorem because the number of indeterminates in $\mathcal{R}_n$ increases with $n$. However, if the ideals comprising a chain are $\symm$-invariant, we may still be able to find an $N$ such that all the ideals $I_{N}$, $I_{N+1}, \ldots$ are the same.  We now make these notions precise (with corresponding definitions for Laurent ideals and the rings $R_n$). 

\begin{defn}\label{def:invariantChain}
A chain $I_\circ:=I_1\subseteq I_{2}\subseteq \cdots$ of ideals $I_n\subseteq \mathcal{R}_n$ is an \emph{invariant chain} if \[ \symm_m I_n  \subseteq I_m, \ \ \ \text{for all $m \geq n.$}\]
\end{defn}

\begin{defn}\label{def:stabilization}
An invariant chain $I_\circ$ \emph{stabilizes} if there is an integer $N$ such that \[\< \symm_m  I_N \>_{\mathcal{R}_m} = I_m, \ \ \ \text{for all $m \geq  N.$}\]  Such an $N$ is a \textit{stabilization bound} for the chain, and generators for $I_N$ are called \textit{generators} for $I_{\circ}$.
\end{defn}

In words, an invariant chain stabilizes when its fundamental structure is contained in a finite number of ideals comprising the chain. When $k = 1$, every invariant chain of ideals in $\{\mathcal{R}_n\}_{n\in \P}$ stabilizes \cite{AH07, HS}.  However, the corresponding fact  fails to hold for $k \geq 2$ (e.g., see \cite[Proposition 5.2]{AH07} or \cite[Example 3.8]{HS}), and more refined methods are required to detect chain stabilization.

In many applications, the invariant chains consist of toric ideals, so we shall focus our attention here on the slightly more general class of lattice ideals (see Section \ref{genLaurentThm} for definitions).  For instance, the independent set conjecture in algebraic statistics \cite[Conj.  4.6]{Hosten2007} concerns stabilization for a large family of toric chains.   

Our first main result asserts that invariant chains of lattice ideals stabilize locally, and it is similar to a chain stabilization result used in a recent proof  \cite{HS} of the independent set conjecture. We prove this result in Section \ref{genLaurentThm} using ideas from order theory as described in Section \ref{prelim}.

\begin{thm}\label{laurentlatticethm}
Every invariant chain $I_{\circ}^{\pm} := I_1^{\pm} \subseteq I_{2}^{\pm} \subseteq \cdots$ of Laurent lattice ideals $I_n^{\pm} \subseteq \mathcal{R}_n^{\pm}$ (resp. $I_n^{\pm} \subseteq R_n^{\pm}$) stabilizes.
\end{thm}

Although this result is quite general, our proof is nonconstructive.  In applications, however, one usually desires bounds on chain  stabilization.  Our second main result restricts to the rings $R_n$ and provides a stabilization bound for the special case of Laurent toric chains induced by a monomial \cite[Section 5.2]{AH07}, which we study in Section \ref{stablaurchainmon}.  These toric ideals appear in applications to algebraic statistics \cite{GMV10, HS} and voting theory~\cite{DEMO09}. 

\begin{thm}\label{mainthm2}
Let $f\in \K[y_1,\ldots,y_k]$ be a monomial of degree $d$ in $k$ variables. For each $n\geq k$,  
consider the (toric) map: \[\phi_n : R_n \to T_n,\hspace{10pt}  
x_{(u_1,\ldots,u_k)}\mapsto f(t_{u_1},\ldots,t_{u_k}).\]  Let $I_n= 
\ker \phi_n$, and let $I_n^\pm$ be the corresponding Laurent ideal.  
Then $N = 2d$ is a stabilization bound for the invariant chain $I_\circ^\pm =  I_k^{\pm}  
\subseteq I_{k+1}^{\pm} \subseteq \cdots$ of Laurent ideals.
\end{thm}

\begin{ex}\label{ex:2ndThm}
Let $k=2$ and suppose that $f=y_1^2y_2\in\K[y_1,y_2]$. For every $n\geq 2$, the map $\phi_n$ is defined by $\phi_n(x_{(i,j)})=t_i^2t_j$ for $(i,j)\in\< n\>^2$. Theorem \ref{mainthm2} asserts that if $N=2\cdot \deg(f)=6$, then the generators of $I_6^\pm$ form a generating set for the whole chain $I_\circ^\pm$ up to the action of the symmetric group $\symm_m$; that is, for all $m\geq 6$, we have $\<\symm_m  I_6^\pm\>_{R_m}=I_m^\pm$. For instance, when $m\geq 9$, we observe that $x_{(3,9)} x_{(7,9)}-x_{(3,7)} x_{(9,7)}\in I_m$ (thus, in $I_m^\pm$) since 
$$\phi_n(x_{(3,9)} x_{(7,9)}) = t_3^2 t_7^2 t_9^2 = \phi_n(x_{(3,7)} x_{(9,7)}).$$
Thus, by Theorem \ref{mainthm2}, there exist permutations $\sigma_1,\ldots,\sigma_r\in \symm_m$, elements $g_1,\ldots,g_r\in I_6^\pm$, and polynomials $h_1,\ldots, h_r\in R_m^\pm$, such that $x_{(3,9)} x_{(7,9)}-x_{(3,7)} x_{(9,7)} = h_1 \sigma_1 g_1 + \cdots + h_r \sigma_r g_r$. Theorem \ref{ouralgthm} below, provides a method for finding such polynomial combinations in general; in this case, one possibility is $r=1$, $h_1=1$, $\sigma_1 = (1\, 3\, 9\, 2\, 7)\in\symm_m$, and $g_1 = x_{(1,3)} x_{(2,3)}-x_{(1,2)} x_{(3,2)}\in I_6^\pm$. For more details on this example (including an explicit set of generators for $I_6^{\pm}$), see Section \ref{stablaurchainmon}.
\qed
\end{ex}

\begin{rem}
Rather surprisingly, it is still an open question whether the (non-Laurent) toric chain $I_{\circ}$ stabilizes in Example \ref{ex:2ndThm}, and more generally, for any monomial $f$ that is not square-free.  Section \ref{openprobconj} discusses more open problems of this nature.
\end{rem}

In the development of the proof of Theorem \ref{mainthm2}, we also found an algorithm for computing these generators.

\begin{thm}[Algorithm \ref{algorithmourthm}]\label{ouralgthm}
There is an effective algorithm to compute a finite set of generators for the Laurent chains $I_\circ^\pm$  in Theorem \ref{mainthm2}.
\end{thm}

The first step of the algorithm in Theorem \ref{ouralgthm} is to embed a toric ideal into a Veronese ideal in a larger polynomial ring and use the fact that the latter is generated by quadratic binomials. A second procedure replaces the extra indeterminates of the larger ring by special quotients of monomials involving only indeterminates of the original polynomial ring.  In turn, this reduces to an integer programming problem, which we solve explicitly.  The following example illustrates some of the main ideas involved.

\begin{ex}
(Continuing Example \ref{ex:2ndThm}). 
Consider the polynomial rings $R_n^\prime :=R_n[x_{(1,2,3)}]$ in an extra indeterminate $x_{(1,2,3)}$, and extend $\phi_n$ to a map $\phi_n^\prime: R_n^\prime\to T_n$ by setting $\phi_n^\prime(x_{(1,2,3)})=t_1t_2t_3$. Notice that if $h \in I_n$, then $h \in \ker\phi_n^\prime$, and also that 
$$
\phi_n^\prime(x_{(1,2,3)}^2) = \phi_n^\prime(x_{(1,3)} x_{(2,3)}) =  \phi_n^\prime(x_{(1,2)} x_{(3,2)}) = t_1^2 t_2^2 t_3^2.
$$
Thus, $p_1:=x_{(1,3)} x_{(2,3)} - x_{(1,2,3)}^2$ and $p_2:=x_{(1,2)} x_{(3,2)} -x_{(1,2,3)}^2$ lie in $\ker\phi_n^\prime$ (for $n \geq 3$).  Consider any generating set for $\ker\phi_n^\prime$ which contains $p_1,p_2$; then, each $g \in I_n$ can be expressed in terms of these generators. For instance,
\[ g = x_{(1,3)} x_{(2,3)}-x_{(1,2)} x_{(3,2)} = (x_{(1,3)} x_{(2,3)} - x_{(1,2,3)}^2) - (x_{(1,2)} x_{(3,2)} -x_{(1,2,3)}^2)\in \ker\phi_n^\prime.\]
Next, notice that 
\begin{equation}\label{example_subst}
\phi_n^\prime(x_{(1,2,3)}) = t_1t_2t_3 = \frac{\phi_n(x_{(1,2)}) \phi_n(x_{(3,1)})}{\phi_n(x_{(1,3)})} = \phi_n \left( \frac{x_{(1,2)}x_{(3,1)}}{x_{(1,3)}} \right).
\end{equation}
Therefore, if we replace $x_{(1,2,3)}$ by $\frac{x_{(1,2)} x_{(3,1)}}{x_{(1,3)}} $ in the two generators $p_1$ and $p_2$ above, we obtain two elements $\hat{p}_1, \, \hat{p}_2 \in I_n^\pm$ which also generate $g$.  More generally, if we can find a finite set of generators for the chain of ideals $\ker \phi_n'$, then we would have generators for the chain of ideals $I_n$ up to monomial inversion.

Identity (\ref{example_subst}) was discovered by solving the following integer programming problem (described more fully in Example \ref{integer combination}). The exponent vector of $t_1t_2 t_3$ is $u=(1,1,1,0,\ldots,0)\in\Z^n$ and for any $(i,j)\in\<n\>^2$, the exponent vector of $\phi_n(x_{(i,j)}) = t_i^2 t_j$ is \[w_{i,j}:=(0,\ldots, 0, 2,0,\ldots,0,1,0,\ldots,0)\in\Z^n,\] in which the nonzero components of  $w_{i,j}$ are the $i$th and $j$th with respective values 2 and 1. To find an expression such as (\ref{example_subst}), we needed to write $u$ as an integer linear combination of the vectors $w_{i,j}$ (this is done in general in Lemma \ref{columnspan}).  
\qed
\end{ex}

The most recent finiteness result along the lines of Theorems \ref{laurentlatticethm} and \ref{mainthm2} can be found in the work of Draisma and Kuttler~\cite{DraismaKuttler11}.  There, they prove set-theoretically that for any fixed positive integer $r$, there exists $d\in\N$ such that for all $p\in \N$, the set of $p$-tensors (elements of $V_1\otimes\cdots\otimes V_p$, where each $V_i$ is a finite dimensional $\K$-vector space) of border rank at most $r$ are defined by the vanishing of finitely many polynomials of degree at most $d$ (when $r=1$ these polynomials define toric ideals).  The authors of \cite{DraismaKuttler11} also realized the ideals defined by these polynomial equations as invariant chains under the action of the semi-direct product of $\symm_p$ with the general linear group $GL(V)^p$, and they conjectured~\cite[Conjecture 7.3]{DraismaKuttler11} stabilization. 
The case $r=1$ was proved by Snowden in \cite{Snowden10}.
The results of \cite{DraismaKuttler11}  extend those of Landsberg and Manivel in ~\cite{LM2004}, where they show set-theoretically that $p$-tensors of rank at most $2$ are defined by polynomials of degree $3$ (the $(3\times 3)$-subdeterminants of all the \emph{flattenings}) regardless of the dimension of the tensor. We note that an ideal-theoretic proof of this last fact was recently discovered by Raicu~\cite{Raicu2010}.


While the general problem of deciding which chains of ideals stabilize seems difficult, it is possible that every invariant chain of (non-Laurent) lattice or toric ideals stabilizes, and Theorem \ref{laurentlatticethm} provides evidence. However, even for the special case studied here of a toric chain induced by a monomial, this is not known \cite[Conjecture 5.10]{AH07} and appears to be a difficult problem (although it is true for square-free monomials \cite[Theorem 5.7]{AH07}). We pose the following open question.

\begin{prob}
Does every invariant chain of lattice ideals (resp. toric ideals) stabilize?
\end{prob}

The outline of this paper is as follows.  In Section \ref{prelim}, we introduce the order theory required for proving Laurent lattice stabilization (Theorem \ref{laurentlatticethm}) in Section \ref{genLaurentThm}.  Next, Section \ref{stablaurchainmon} contains a proof of Theorem \ref{mainthm2} using some ideas from toric algebra and integer programming.  Also found there is another approach to constructing Laurent chain generators in Theorem \ref{mainthm2} (e.g., the generators alluded to in Example \ref{ex:2ndThm}) which can produce smaller generating sets than those found by Algorithm \ref{algorithmourthm}.  Section \ref{algorithms} contains a discussion of Theorem \ref{ouralgthm} and Algorithm \ref{algorithmourthm}.  Finally, in Section \ref{openprobconj} we present some open problems and conjectures arising from our computational investigations.

\section{Nice Orderings}\label{prelim}

In this section, we explain the ideas from the theory of partial orderings that are needed to prove Theorem \ref{laurentlatticethm}.  A \textit{well-partial-ordering} $\leq$ on a set $S$ is a partial order such that (i) there are no infinite antichains and (ii) there are no infinite strictly decreasing sequences. One can check that this naturally generalizes the notion of ``well-ordering" to orders $\leq$ which are not total.
 
Let $\symm$ be a group acting on a set $S$ (a \textit{$\symm$-set}), and suppose that $\leq$ is a well-ordering of $S$.  For $s\in S$ and $\sigma \in \symm$, let $s_< := \{t\in S:t<s\}$ and $\sigma  s_< := \{\sigma t: t < s\}$.  We define a partial ordering $\preceq$ on $S$ as follows:
\begin{equation}\label{preceqdef}
s\preceq t\quad
:\Longleftrightarrow\quad\text{$s\leq t$ and there exists $\sigma \in \symm$ such that $\sigma s=t$ and $\sigma  s_< \subseteq t_<$.}
\end{equation}
A group element $\sigma \in \symm$ verifying (\ref{preceqdef}) is called a \textit{witness} of the relation $s\preceq t$.  An example of this construction can be found in Example \ref{orders example}. 

Call the well-ordering $\leq$ of $S$ a \emph{nice} ordering if $\preceq$ is a well-partial-ordering.   Many naturally occurring $\symm$-sets have nice orderings.  For instance, the set of $k$-element subsets of $\P$ with the natural action of $\symm = \symm_{\P}$ has a nice ordering \cite{Ahlbrandt-Ziegler}.  Camina and Evans studied the ring-theoretic consequences of nice orderings in \cite{Camina-Evans}, inspired by the ideas in \cite{Ahlbrandt-Ziegler}.  They showed that if $S$ has a nice ordering, then the $\K[\symm]$-module $\K S$ is Noetherian over the group ring $\K[\symm]$ for any field $\K$  \cite[Theorem~2.4]{Camina-Evans}.  We shall prove that $[\P]^k$ also has a nice ordering; however, our application (Theorem \ref{laurentlatticethm}) requires a more refined version of this statement. This refinement is given by Theorem \ref{mainNoethThm} below.  Before proving this theorem, we first define a nice ordering of $[\P]^k$ with special properties.

Consider $\symm_{\P}$ acting on $[\P]^k$ as described in (\ref{SymmPk}).  We first give a total well-ordering $\leq_{dlex}$ on $[\P]^k$ as follows.  Given $w=(w_1,\dots,w_k)\in[\P]^k$, set $|w|_\infty:=\max\{w_1,\dots,w_k\}$.  Define the \textit{degree lexicographic} total ordering on $[ \P ]^k$ by 
\begin{equation}\label{niceorderingPk}
v\leq_{dlex} w \quad:\Longleftrightarrow\quad \text{$|v|_\infty<|w|_\infty$ or $|v|_\infty=|w|_\infty$ and $v <_{lex}w$.}
\end{equation}
Here, $<_{lex}$ is the natural \textit{lexicographic} ordering of elements of $[\P]^k$ given by $(u_1,\ldots,u_k) <_{lex} (w_1,\ldots,w_k) :\Longleftrightarrow u_1 = w_1, \ldots,u_{r-1} = w_{r-1}$ and $u_{r} < w_{r}$ for some $r \in [k]$.

Notice that for every $w\in[\P]^k$ there are only finitely many $v\in[\P]^k$ such that $v<_{dlex}w$; hence, $\leq_{dlex}$ is a well-ordering of $[\P]^k$.  The well-ordering $\leq_{dlex}$ induces the partial order $\preceq_{dlex}$ as in (\ref{preceqdef}).

\begin{ex}\label{orders example}
With the above definition of $\leq_{dlex}$ for $[\P]^2$, we have $(2,3)\leq_{dlex} (2,4)$ and $(2,3)\leq_{dlex} (3,1)$. Moreover, when $[\P]^2$ is equipped with the action of $\symm_\P$, we claim that $(2,3)\preceq_{dlex} (2,4)$.  Represent the elements of $\symm_\P$ in cyclic notation so that $(3\,4)\cdot (2,3) = (2,4)$, and observe that
\begin{eqnarray*}
(3\,4)\cdot(2,3)_{<_{dlex}}&=&(3\,4)\cdot\{(1,3), (2,1), (2,2), (1,2) \}\\
&=&\{(1,4), (2,1), (2,2), (1,2)\}\\
&\subseteq&\{(1,1), (1,4),  (3,1), (1,3), (2,3), (3,2),(3,3), (2,1),(2,2), (1,2)\}\\
&=&(2,4)_{<_{dlex}} .
\end{eqnarray*}
On the other hand, we have $(2,3)\npreceq_{dlex} (3,1)$.  To see this, let $\sigma\in\symm_\P$ be such that $\sigma\cdot(2,3)=(3,1)$; thus  $\sigma(1)\geq 2$. Notice that for $(2,1)\in(2,3)_{<_{dlex}}$, we have $\sigma\cdot(2,1) = (\sigma(2),\sigma(1)) = (3,\sigma(1))$.  It follows that $(3,2)\leq_{dlex}\sigma\cdot(2,1)$.   Since $$(3,1)_{<_{dlex}} =\{(2,3), (1,3), (2,1), (1,2),(1,1),(2,2) \},$$
we see that $(3,2)\notin (3,1)_{<_{dlex}}$; therefore, $\sigma\cdot(2,3)_{<_{dlex}} \nsubseteq (3,1)_{<_{dlex}}$. 
\qed
\end{ex}

Although not needed for our main result, a solution to the following problem would likely be useful in converting the methods of this section into computational tools.  

\begin{prob}
Give a computationally efficient criteria for determining if $u \preceq_{dlex} v$ for $u,v \in [\P]^k$.
\end{prob}

One may also ask the following open-ended problem.
 
\begin{prob}
Let $S$ be an $\symm$-set.  Characterize those total well-orderings $\leq$ which are nice.
\end{prob}

We are now in position to show that the ordering $\leq_{dlex}$ is nice.

\begin{prop}\label{niceOrdProp}
The ordering $\preceq_{dlex}$ of $[\P]^k$ is a well-partial-ordering.
\end{prop}

The proof of this proposition uses a special case of a result of Higman \cite{Hig52, NW1}, which we state in the following lemma.
Recall that a \textit{strictly increasing} map $\varphi: [m] \to [n]$ satisfies $\varphi(i) < \varphi(i+1)$ for all $i$.

\begin{lem}[\cite{Hig52}]\label{HigLemma}
Let $\Sigma$ be a finite set. The following ordering $\leq_H$ on the set $\Sigma^*$  of all finite sequences of
elements of $\Sigma$ is a well-partial-ordering:
$$(x_1,\dots,x_m) \leq_H (y_1,\dots,y_n) \quad :\Longleftrightarrow \quad
\begin{cases}
&\text{\parbox{160pt}{ $\exists\, \, \varphi\colon
[m]\to[n]$ such that $\varphi$ is strictly increasing and $x_i = y_{\varphi(i)}$ for
all $i\in[m]$}}\end{cases}
$$
\end{lem}

\begin{proof}[Proof of Proposition \ref{niceOrdProp}]
Let $\Sigma:=\{0,1,\dots,k\}$.   First order $[\P]^k \times \Sigma^*$ by the product of the orderings $\leq_{dlex}$ and $\leq_H$ on $[\P]^k$ and $\Sigma^*$, respectively. Then $[\P]^k\times \Sigma^*$ is well-partial-ordered, by Higman's Lemma (the product ordering of two well-partial-orderings is a well-partial ordering). For  $w=(w_1,\dots,w_k)\in[\P]^k$, set $n:=|w|_\infty$; also, let $w^*:=(w_1^*,\dots,w_n^*)\in\Sigma^*$ be given by  
$$
w_i^* := \sum_{w_j=i} j \ \ \ \ \text{for $i=1,\dots,n$}
.$$
To prove that $\preceq_{dlex}$ is a well-partial ordering on $[\P]^k$, it suffices to show that the map $w\mapsto (w,w^*)\colon [\P]^k\to [\P]^k\times\Sigma^*$ is an \textit{order-embedding}; that is, if $v\leq_{dlex} w$ and $v^*\leq_H w^*$, then $v\preceq_{dlex} w$ for all $v,w\in[\P]^k$. 

Suppose that $v\leq_{dlex} w$ and $v^*\leq_H w^*$, and let $m = |v|_\infty, \, n=|w|_\infty$; then there exists a function $\varphi\colon [m]\to[n]$ strictly increasing such that $v_i^* = w^*_{\varphi(i)}$ for all $1\leq i \leq m$. Since $\varphi$ is injective, it can be extended to a permutation $\sigma\in\symm_\P$. We claim that $v\preceq_{dlex} w$ via witness $\sigma$ so that $\sigma v=w$ and $\sigma v_{<_{dlex}} \subseteq w_{<_{dlex}}$.  

We first verify that $\sigma v=w$.  For $i \in \{1,\ldots,k\}$, let $l=v_i\leq m$. Notice that $w^*_{\sigma(l)}=v^*_l$, and so together with the definition of $v^*$, we have $$v^*_l =i+\sum_{\underset{j\neq i}{v_j=l}}j \Longrightarrow w^*_{\sigma(l)} = i+\sum_{\underset{j\neq i}{w_j=\sigma(l)}}j.$$  In particular, $w_i = \sigma(l)$; thus, $\sigma(v_i) = \sigma(l) = w_i$ and so $\sigma v = w$.

Now, suppose $u\leq_{dlex} v$. Since $\sigma$ and $\varphi$ agree on $\{v_1,\ldots,v_k\}$, it follows that $|\sigma u|_\infty \leq |\sigma v|_\infty = |w|_\infty = n$. To show $\sigma u \leq_{dlex} w$, it suffices to verify this when $|u|_\infty = |v|_\infty$, as the other case follows from $\varphi$ being strictly increasing. If $|u|_\infty= m = |v|_\infty$ and $u\leq_{dlex} v$, there is an $r\in[k]$ such that $u_1=v_1,\ldots,u_{r-1}=v_{r-1}$ and $u_r<v_r$. Therefore, $\sigma(u_1) = w_1,\ldots,\sigma(u_{r-1})=w_{r-1}$ and $\sigma(u_r)<\sigma(v_r)=w_r$ as $\sigma$ is strictly increasing. Thus, $\sigma u\leq_{dlex}w$ and so $\sigma v_{<_{dlex}} \subseteq w_{<_{dlex}}$ as required.
\end{proof}

\begin{rem}
Higman's lemma is also a key element in all known proofs of the finiteness result for $\symm_{\P}$-invariant ideals of $\mathbb C[x_1,x_2,\ldots]$ that was mentioned in the introduction.
\end{rem}

The following result also follows from the proof of Proposition \ref{niceOrdProp}.

\begin{cor}
The ordering $\preceq_{dlex}$ of $\<\P\>^k$ is a well-partial-ordering.
\end{cor}

\begin{proof}
The same proof as Proposition \ref{niceOrdProp} works just by noticing that, in this case, $w_i^*=j$ if $w_j=i$ or $0$ otherwise.
\end{proof}


Not all natural orders are nice as the following example demonstrates.

\begin{ex}
Define the \emph{reverse lexicographic ordering} $\leq_{revlex}$ on $[\P]^k$ as follows: 
\begin{multline}\label{revLex}
\qquad(u_1,\ldots,u_k)\leq_{revlex} (w_1,\ldots,w_k) \ :\Longleftrightarrow \ u_k=w_k, \ldots, u_{k-r}=w_{k-r} \\ \text{ and } \ w_{k-r-1} < u_{k-r-1} \, , \qquad
\end{multline}
for some $r\in[k]$.  In contrast to Proposition \ref{niceOrdProp}, the partial order $\leq_{revlex}$ is not nice. For instance, we have in $[\P]^2$ the following infinite strictly decreasing sequence:
\[\cdots \preceq_{revlex}(6,3)  \preceq_{revlex}(5,3) \preceq_{revlex}(4,3). \qed\]
\end{ex}

The nice ordering $\preceq_{dlex}$ is useful theoretically because of the following property.

\begin{lem}\label{holefixer}
Let $\preceq_{dlex}$ be the well-partial-ordering (\ref{preceqdef}) induced by the nice ordering $\leq_{dlex}$ of $[\P ]^k$. Also, let $s,\,t\in[\P]^k$ satisfy $s \preceq_{dlex} t$ and $|t|_{\infty} \leq M$ for some $M\in \{0,1,\ldots\}$. Then there is a $\sigma \in \symm_{M}$ witnessing $s \preceq_{dlex} t$.
\end{lem}

\begin{proof}
Since $s\preceq_{dlex} t$, there exists $\tau\in\symm_\P$ such that $\tau s = t$ and $\tau  s_{<_{dlex}} \subseteq t_{<_{dlex}}$.
Let $M = |t|_\infty$.  Construct $\sigma\in\symm_M$ by setting $\sigma(i):= \tau(i)$ if $\tau(i)\leq M$ and then extending $\sigma$ to a permutation of $[M]$.  We claim that $\sigma s = t$ and $\sigma  s_{<_{dlex}}\subseteq t_{<_{dlex}}$. Since $s\leq_{dlex} t$, we have $|s|_\infty \leq |t|_\infty = M$.  Therefore writing $s = (s_1,\ldots,s_k)\in [\P]^k$, it follows that $\tau(s_i)\leq
M$ for each $i$; thus, $\sigma(s)=\tau(s)=t$. Notice also that $\tau(w)\leq M$ for all $w\in s_{<_{dlex}}$ because for all $w\in s_{<_{dlex}}$, we
have $w<_{dlex}s$ which implies $|w|_\infty \leq |s|_\infty\leq M$, and the same holds for all
$u\in t_{<_{dlex}}$. Therefore, $\tau(w_i)\leq |t|_\infty = M$ for all $w\in s_{<_{dlex}}$ and each $i=1,\ldots,k$. Thus, $\sigma(w) = \tau(w)$; therefore, $\sigma s_{<_{dlex}}\subseteq t_{<_{dlex}}$.
\end{proof}

If $A$ is a commutative ring and $S$ an $\symm$-set, we let $AS$ denote the free $A$-module with basis $S$. Also, let $A[\symm]$ be the (left) group ring (whose elements are formal linear combinations of elements in $\symm$ with coefficients in $A$~\cite{Lam2001}). The natural linear action of $A[\symm]$ on $AS$ makes it into an $A[\symm]$-module. The following is the refinement of the Noetherianity result from \cite{Camina-Evans} that we will use to prove Theorem \ref{laurentlatticethm}. 

\begin{thm}\label{mainNoethThm}
Let $A$ be a Noetherian commutative ring.  For every $A[\symm_{\P}]$-submodule \mbox{$B \subseteq A[\P]^k$}, there exists a finite 
set $G \subseteq B$ such that  \[ f \in B \cap A[m]^k \   \Longleftrightarrow \ \exists \, \sigma_1,\ldots,\sigma_{\ell} \in \symm_m; \, g_1,\ldots,g_{\ell} \in G; \, a_1,\ldots,a_{\ell} \in A   \text{ with } \ f = \sum_{i=1}^{\ell} a_i \sigma_i g_i.\]
\end{thm}

\begin{proof}
Let $\preceq_{dlex}$ be the well-partial-ordering of $[\P]^k$ (by Proposition \ref{niceOrdProp}) induced by the total well-order $\leq_{dlex}$ from (\ref{niceorderingPk}).  A \textit{final segment} of the partial order  $\preceq_{dlex}$ is a set $F \subseteq [\P]^k$ such that $u \in F$ and $u \preceq_{dlex} v$ implies that $v \in F$.  A well-known characterization of well-partial-orderings (see e.g. \cite{Kruskal72}) is that final segments are finitely generated.  That is, for every final segment $F$, there is a finite set $T \subseteq F$ such that $F = \{ v : \exists \, u \in T \text{ with } u \preceq_{dlex} v\}$.

If $f \in A[\P]^k$, we define the \emph{head} of $f$, Head$(f)$, to be the largest nonzero element in $[\P]^k$ (with respect to $\leq_{dlex}$) in the \textit{support} of $f$ (those elements of $[\P]^k$ occurring in $f$ with nonzero coefficient).

For the $A[\symm_\P]$-submodule $B$, let $J \subseteq A$ be the ideal generated by the (leading) coefficients of Head($f$) as $f$ ranges over elements of $B$.  By Noetherianity of $A$, we have $J = \<c_1,\ldots,c_r\>_{A}$ for some $c_i \in A$.  Also, since $\preceq_{dlex}$ is a well-partial-order, the final segment $F = \{\text{\rm Head}(f) : f \in B\}$ is finitely generated by $T = \{ \text{\rm Head}(b_1), \ldots, \text{\rm Head}(b_{|T|})\}$ for some $b_j \in B$.  Consider now the finite set, \[ G := \{ c_i b_j :  \,  1 \leq i \leq r, \, 1 \leq j \leq |T|\} \subseteq B.\]
We claim that $G$ is a subset of $B$ fulfilling the requirements of the theorem statement.  

Let $f  \in B \cap A[m]^k$.  Then, Head($h_1) \preceq_{dlex}$ Head($f$) for some $h_1 \in \{b_1,\ldots,b_{|T|}\}$ with witness $\sigma_1 \in \symm_m$ (by Lemma \ref{holefixer}).  There are $a_1,\ldots,a_r \in A$ such that \[ f_1 := f - \sum_{i = 1}^r a_i c_i \sigma_1 h_1 \in B\] has a strictly smaller (with respect to $\leq_{dlex}$) head than $f$.  Continuing in this manner we can produce a sequence $f_1,f_2,\ldots$ of elements in $B$ such that \[ \cdots  \leq_{dlex} \text{\rm Head}(f_2) \leq_{dlex} \text{\rm Head}(f_1) \leq_{dlex}  \text{\rm Head}(f).\]
Since $\leq_{dlex}$ is a well-ordering, it follows that $f_{p} = 0$ for some $p \in \P$ which gives an expansion for $f$ as in the statement of the theorem.
\end{proof}

\begin{cor}\label{NoethCor}
$A[\P]^k$ and $A\<\P\>^k$ are Noetherian $A[\symm_{\P}]$-modules.
\end{cor}

\begin{rem}
It turns out that Corollary \ref{NoethCor} holds when $A\<\P\>^k$ is replaced by $AS$ and  $A[\symm_{\P}]$ by  $A[\symm]$ for any $\symm$-set $S$ with a nice ordering (this follows from the argument above).  However, to prove Theorem \ref{laurentlatticethm} in the next section, we need the more refined statement found in Theorem \ref{mainNoethThm}, which asks for witnesses $\sigma$ to (\ref{preceqdef}) having  special properties.
\end{rem}

\section{Laurent chain stabilization}\label{genLaurentThm}

In this short section, we prove that invariant chains of Laurent lattice ideals stabilize (this is Theorem \ref{laurentlatticethm} from the introduction). The proof uses the order theory from the previous section and a few properties of lattice ideals. Some basic material on lattice and toric ideals can be found in \cite[Chapter 7]{CCA} and \cite{SturmfelsGBCP}, respectively, and  a more general reference for binomial ideals is \cite{EisSturm96}.

Let $\mathcal{G}$ be a finitely generated abelian group and let $a_1,\ldots,a_d$ be distinguished
generators of $\mathcal{G}$. Let $L$ denote the kernel of the surjective homomorphism $\mathbb Z^d$ onto $\mathcal{G}$.
 The \textit{lattice ideal} associated with $L$ is the following ideal in $\K[z_1,\ldots,z_d]$: \[I_L = \langle {\bf z}^ u - {\bf z}^v : u, v \in \N^d \text{ with }  u-v \in L \rangle.\]
Here,  we use the shorthand ${\bf z}^u = z_1^{u_1} \cdots z_d^{u_d}$ for $u = (u_1,\ldots,u_d) \in \Z^d$.  A \textit{toric ideal} is the special case of a lattice ideal in which the group $\mathcal{G}$ is torsion-free; in this case, the ideal $I_L$ is also a prime ideal. 

Notice that if $S = \{s_1,\ldots,s_d\}$ is a set with $d$ elements, there is a natural isomorphism between $\Z^d$ and the free $\Z$-module $\Z S$ with basis $S$ given by:
$$(a_1,\ldots,a_d)\in \Z^d  \mapsto  \sum_{i=1}^d a_i s_i\in \Z S.$$
Although simple, this identification will be useful for us below.

\begin{ex}
In the case $S=\<3\>^2=\{(1,2),(1,3),(2,1),(2,3),(3,1),(3,2)\}$, the integer vector $(-2,2,1,0,-1,0)\in\Z^6$ is also represented by $-2\cdot(1,2)+2\cdot(1,3)+(2,1)-(3,1)\in \Z\<3\>^2$.
\qed
\end{ex}

For simplicity of exposition, we focus our attention on lattice ideals in the polynomial rings $\mathcal{R}_n$ (equipped with the action of $\symm_n$) from (\ref{Rndefs}), each of which has $d_n = n^k$ indeterminates.  Let $L_n \subseteq L_{n+1}$ be an increasing sequence of subgroups of $\Z^{d_n} \subseteq \Z^{d_{n+1}}$ and let $I_n := I_{L_n} \subseteq \mathcal{R}_n$ (resp. $I_n^{\pm} \subseteq \mathcal{R}_n^{\pm}$) be the corresponding lattice (resp. Laurent lattice) ideals. 

The basic idea in our proof of Theorem \ref{laurentlatticethm} is to view $L = \bigcup_{n \in \P} L_n$ as an $\symm_{\P}$-invariant subgroup of the free abelian group $\Z [\P]^k = \bigcup_{n \in \P} \Z [n]^k$, which has free basis $[\P]^k$ over $\Z$.  The set $L$ has a finite generating set up to $\symm_{\P}$-symmetry (using Theorem \ref{mainNoethThm} and the fact that $\Z$ is Noetherian), and these vectors are all contained in $L_N$ for some integer $N$.
The remainder of the proof converts this fact back to the level of ideals.  The complete details are as follows.

Given an integer vector $h \in \Z^{d}$, we set $h_+ \in \N^{d}$ and $h_{-} \in \N^{d}$ to be the nonegative and nonpositive part of $h$, respectively (so that $h = h_+ - h_{-}$).  The following is elementary.

\begin{lem}\label{laurentLemma}
Suppose that $v, h_1,\ldots,h_m \in \Z^{d}$ and set $u = v + \sum_{i=1}^m{h_i}$.  There exists a monomial ${\bf z}^c \in \K[z_1^{\pm 1},\ldots,z_d^{\pm 1}]$ such that 
${\bf z}^c ({\bf z}^u - {\bf z}^v) \in \< {\bf z}^{{h_i}_+} - {\bf z}^{{h_i}_{-}}: i = 1,\ldots,m \>_{\K[z_1^{\pm 1},\ldots,z_d^{\pm 1} ]}.$
\end{lem}
\begin{proof}
We shall induct on $m$, the base case being vacuously true.  Consider the identity:
\begin{equation}
\begin{split}
({\bf z}^u - {\bf z}^v) = {\bf z}^{h_m} ({\bf z}^{u-h_m} - {\bf z}^v) + {\bf z}^{v - {h_m}_{-}}({\bf z}^{{h_m}_{+}} - {\bf z}^{{h_m}_{-}}). \\
\end{split}
\end{equation}
As $u' = u - h_m$ has fewer terms, the proof follows by induction.
\end{proof}

Collecting these facts together, we can now prove the main result of this section.

\begin{proof}[Proof of Theorem \ref{laurentlatticethm}]
The submodule $L \subseteq \Z [\P]^k$ is finitely generated over $\Z [\symm_{\P}]$ by Theorem \ref{mainNoethThm}.  Set $H =  G \cup -G$ for a finite set of generators $G \subseteq L$ satisfying the property in Theorem \ref{mainNoethThm}, and let $N$ be such that $H \subseteq \Z^{d_N}$.  Consider two vectors $u,v \in \N^{d_m}$ such that $u - v \in L_m=L\cap [m]^k$, with $m \geq N$.   By assumption, the vector $u - v$ is a $\Z$-linear combination of $\symm_m$-permutations of elements in $H$.   By Lemma \ref{laurentLemma}, 
it follows that  ${\bf z}^u - {\bf z}^v$ is a monomial multiple of an element in the ideal (of $\mathcal{R}_m^\pm$) generated by permutations (in $\symm_m$) of $\{{\bf z}^{h_+} - {\bf z}^{h_{-}} : h \in H\}$.  Thus, $I_m^{\pm} \subseteq \<\symm_m I_N\>_{\mathcal{R}_m^{\pm}}$ and the chain stabilizes.
\end{proof}

\section{Stabilization of chains induced by monomials}\label{stablaurchainmon}

We now focus on the polynomial rings $R_n$ from (\ref{Rndefs}) and the corresponding chains of toric ideals encountered in the statement of Theorem \ref{mainthm2}.  

\begin{defn}\label{ker_ind_poly}
Let $k\in\P$ and $f\in\K[y_1,\ldots, y_k]$. For for each $n\geq k$, consider
\[ \phi_n : R_n\to T_n,\hspace{15pt} x_{(u_1,\ldots,u_k)}\mapsto f(t_{u_1},\ldots,t_{u_k}).\]
Let $I_n=\ker \phi_n$. The invariant chain $I_{k}\subseteq I_{k+1}\subseteq \cdots$ is called the invariant chain of ideals \emph{induced by the polynomial $f$}.
\end{defn}

The ideals in Definition \ref{ker_ind_poly} appear in voting theory~\cite{DEMO09}, algebraic statistics~\cite{SturmfelsSullivant05, HS, Draisma2010, GMV10}, and toric algebra~\cite{AH07, HS}.  When $f$ is a monomial, each $I_n = \ker \phi_n$ is a homogeneous toric ideal. The following was conjectured in \cite{AH07}. 

\begin{conj}[\cite{AH07}]\label{conjwmonomial}
The chain of ideals induced by any monomial stabilizes.
\end{conj}

The authors of \cite{AH07} verified the special case of Conjecture \ref{conjwmonomial} when $f$ is a square-free monomial.  Underlying their proof is the fact that for every $n \geq k$, the ideals $I_n$ are generated by quadratic binomials \cite[Theorem 14.2]{SturmfelsGBCP}. Unfortunately, the corresponding statement is false when $f$ is not square-free. Although a proof for the general conjecture is not known, Theorem \ref{laurentlatticethm} shows (albeit nonconstructively) that the Laurent versions of these chains stabilize.   

The main goal of this section is to provide an effective version of Theorem \ref{laurentlatticethm} for this situation that allows for explicit computation of generators (this is Theorem \ref{mainthm2} from the introduction).  In the next section, we describe this algorithm and give a reference to an implementation of it in software.  We also explain another approach to finding these generators at the end of this section. 

Our running example throughout will be the case $f=y_1^2 y_2$, and all computations were performed using \texttt{Macaulay2}~\cite{M2} and \texttt{4ti2}~\cite{4ti2}.
If $I_\circ$ is the chain of ideals induced by $y_1^2 y_2$, then Theorem \ref{laurentlatticethm} guarantees stabilization of $I^\pm_\circ$. Moreover, Theorem \ref{mainthm2} provides a stabilization bound $N=2\cdot \deg (f) = 6$.  
Using Algorithm \ref{algorithmourthm} from Section \ref{algorithms}, the following is a generating set for $I^\pm_\circ$ (below, we use a shorthand notation for indices; e.g., $x_{16} = x_{(1,6)}$):

{ \renewcommand{\tabcolsep}{2.5mm}
\renewcommand{\arraystretch}{1.3}
\begin{tabular}{r l l } 
$G^\pm \ =$ &$\big \{x_{16}x_{21}^2x_{54}x_{65}-x_{14}x_{15}x_{26}^2x_{56},$
&$x_{16}^2x_{21}^4x_{43}x_{65}-x_{13}x_{14}^2x_{15}x_{26}^4$,\\
&\ \ $x_{16}^2x_{21}^4x_{45}x_{65}-x_{14}^2x_{15}^2x_{26}^4,$
&$x_{16}x_{21}^2x_{34}x_{65}-x_{14}x_{15}x_{26}^2x_{36},$\\
& \ \ $x_{16}x_{21}^2x_{36}-x_{13}^2x_{26}^2,$
&$x_{16}^2x_{21}^2x_{32}-x_{12}x_{13}^2x_{26}^2$,\\
& \ \ $x_{13}x_{43}-x_{14}x_{34},$
&$x_{13}x_{24}-x_{14}x_{23}\big\}.$ 
\end{tabular}}

Therefore, the chain of Laurent ideals $I^\pm_\circ$ induced by $y_1^2y_2$ is generated by these $8$ elements of $G^\pm$ up to the action of the symmetric group.  It is important to remark that these binomials are not generators of the original ideal $I_6$, nor of the chain $I_\circ$.  Moreover, this generating set is not smallest possible, as shown in Section \ref{sec:themainexample}, where we study the combinatorial structure of this special case and find a generating set with only $4$ elements for the Laurent chain $I^\pm_\circ$.

\subsection{Proof of Theorem \ref{mainthm2}}

First observe that the inclusion $R_n\hookrightarrow R_n^\pm$ gives us for every $n\geq k$ an extension of $\phi_n$ given by the homomorphism $\psi_n: R_n^\pm\rightarrow T_n^\pm$ satisfying $\psi_n(x_u)=\phi_n(x_u)$ and $\psi_n(x_u^{-1})=\phi_n(x_u)^{-1}$ for all $u\in \< n \>^k$. Notice that $I_n^{\pm}=\ker \psi_n$ and that we have the following commutative diagram:  
\begin{equation}\label{com_diagram}
\xymatrix{
R_n \ar@{^{(}->}[r] \ar[rd]_{\phi_n} & R_n^{\pm}\ar[d]^{\psi_n} \\
& T_n^{\pm}}
\end{equation}

Let $\alpha\in \N^k$ be the exponent vector of a (non-constant) monomial $f= {\bf y}^\alpha =y_1^{\alpha_1}\cdots y_k^{\alpha_k}$, and consider 
\[\matrA_n :=\{ \sigma(\alpha_1,\ldots,\alpha_k,0,\ldots,0)^\top \in \Z^n : \sigma\in\symm_n \}.\]
The set of column vectors $\matrA_n$ can be represented as an $n \times {n\choose k}k!$ matrix with rows indexed by the indeterminates $t_i$ (for $i=1,\ldots, n$) and columns indexed by the indeterminates $x_w$ (for $w\in\< n \>^k$).  The matrix $\matrA_n$ defines a semigroup homomorphism $\N^{{n\choose k}k!}\rightarrow \N^n$ which lifts to the homomorphism $\phi_n$. The kernel $I_n$ is generated by the set:
$$
\left\{ {\bf x}^a - {\bf x}^b : \matrA_n(a)=\matrA_n(b), \, a,b \in \N^{{n\choose k}k!}\right\}. 
$$
For more details about toric ideals and their generating sets, see \cite{SturmfelsGBCP}.

\begin{ex}\label{example An}
Let $k=2$, $n=3$, and $\alpha=(2,1)$. The following represents the matrix $\matrA_3$ associated to the homomorphism $\phi_3$ defined by $f=y_1^2y_2$. 
\begin{center}
\begin{tabular}{ccccccc}
 & $x_{12}$ & $x_{13}$ & $x_{21}$ & $x_{23}$ & $x_{31}$ & $x_{32}$ \\
$t_1$ & 2 &2 &1& 0& 1& 0 \\
$t_2$ & 1 &0 &2 &2 &0 &1 \\
$t_3$ & 0 &1 &0 &1 &2 &2
\end{tabular}
\end{center} 
The ideal $I_3$ is generated by binomials:  $\{x_{13}x_{21}^2-x_{12}^2x_{23},\, 
x_{13}^2x_{21}-x_{12}^2x_{31},\,
x_{21}x_{31}-x_{12}x_{32}$, $\,
x_{21}^2x_{32}-x_{12}x_{23}^2,\,
x_{13}x_{23}-x_{12}x_{32},\,
x_{13}x_{21}x_{32}-x_{12}x_{23}x_{31},\,
x_{13}^2x_{32}-x_{12}x_{31}^2,\,
x_{23}x_{31}^2-x_{13}x_{32}^2$, $\,
x_{23}^2x_{31}-x_{21}x_{32}^2\}.$
\qed
\end{ex}

Next, we argue that it suffices to study those maps $\phi_n : R_n\rightarrow T_n$ defined by an exponent vector $\alpha=(\alpha_1,\ldots, \alpha_k)\in\N^k$ with $\gcd(\alpha):=\gcd(\alpha_1,\ldots, \alpha_k)=1$. To see this, suppose that $\gcd(\alpha)=\ell>1$, and consider $\alpha'= \ell^{-1}\cdot\alpha$.  Let $\phi_n$ and $\phi_n'$ be the homomorphisms given by $\phi_n(x_w) = t_{w_1}^{\alpha_1}\cdots t_{w_k}^{\alpha_k}$ and $\phi_n'(x_w) = t_{w_1}^{\alpha_1'}\cdots t_{w_k}^{\alpha_k'}$, respectively. Note that $\phi_n(x_w) = (\phi_n'(x_w))^\ell$ for all $w\in\<n\>^k$, so if $a,\,b\in \N^{{n\choose k}k!}$, then $\phi_n({\bf x}^a) = \phi_n({\bf x}^b) \iff \phi_n'({\bf x}^{ a})^\ell = \phi_n'({\bf x}^{b})^\ell \iff \phi_n'({\bf x}^{ a}) = \phi_n'({\bf x}^{b})$ (as $\phi_n'({\bf x}^a)$ and $\phi_n'({\bf x}^b)$ are monomials in $T_n$); thus, ${\bf x}^a - {\bf x}^b \in\ker\phi_n$ if and only if ${\bf x}^a - {\bf x}^b \in\ker\phi_n'$.

Our first basic tool is a combinatorial lemma describing the $\Z$-linear column span of $\matrA_n$ inside $\Z^n$.
For $\alpha\in\Z^k$, we set $|\alpha| := \sum_{i=1}^k \alpha_i$.

\begin{lem}\label{columnspan}
Let $\alpha = (\alpha_1,\ldots,\alpha_k)\in \N ^k$ with $\gcd(\alpha_1,\ldots,\alpha_k)=1$. 
The integral span of the columns $\matrA_n$ ($n > k$) is:
\[ \spanZ(\matrA_n) = \{ \beta\in\Z^n :\, |\beta|\equiv 0 \mod |\alpha|\}.\]
\end{lem}

\begin{proof}
Let $\mathfrak{A}=\{ \beta\in\Z^n :\, |\beta|\equiv 0 \mod |\alpha|\}$; clearly, $\spanZ(\matrA_n) \subseteq \mathfrak{A}$. By assumption, $\gcd(\alpha)=1$; thus, there are integers $b_1,\ldots,b_k\in \Z$ with $b_1 \alpha_1 +\cdots + b_k \alpha_k = 1$.

For every $j=1,\ldots, n-1$ and $i=1,\ldots,k$, let $\sigma_{ij}$ be the transposition $(i \, j)\in \symm_n$, and consider the vector
$$
h_j=b_1(\sigma_{1j}\alpha)+\cdots +b_k(\sigma_{kj}\alpha).
$$
Notice that $h_j$ is a vector whose $j$th entry is $1$ and $h_j\in \text{Span}\{\matrA_n\}$. Consider also the transposition $\tau_j=(j\, n)$, and the vector 
$$
h_j' =  \tau_j h_j 
= b_1(\tau_j\sigma_{1j}\alpha)+\cdots +b_k(\tau_j\sigma_{kj}\alpha).
$$
For every $i$ and $j$, the composition $\tau_j\sigma_{ij}$ is the transposition $(i\, n)\in\symm_n$; thus, $h'_j$ is obtained from $h_j$ by changing the 1 from position $j$ to position $n$.  Naturally, $h_j'\in \spanZ(\matrA_n)$.  Let $\widehat{h_j}=h_j-h_j' \in \spanZ(\matrA_n)$. Notice that $\widehat{h_j}$ is the vector with 1 in the $j$th position, $-1$ in the $n$th, and zeroes elsewhere.

Now, let $\beta=(\beta_1,\ldots,\beta_n) \in\mathfrak{A}$. By assumption, there exists $q\in\Z$ such that $|{\bf \beta}|=q|\alpha|$. For every $j=1,\ldots,n-1$, there is $r_j\in\Z$ with $\beta_j = q\alpha_j +r_j$. Set $$\gamma=q\alpha + \sum_{j=1}^{n-1} r_j \widehat{h_j}\in \spanZ(\matrA_n) .$$
It is easy to check that $\beta=\gamma$, and so $\beta\in \spanZ(\matrA_n)$ as desired.
\end{proof}

\begin{ex}\label{integer combination}
Consider $\alpha=(2,\, 1)$ and $n=3$. Since $\gcd(2,1)=1$, we can write $1=(1)2+(-1)1$. The vectors $\widehat{h}_1,\widehat{h}_2 \in \text{Span}_{\Z}\{\matrA_3\}$ from the proof of Lemma \ref{columnspan} are precisely $(1,0,-1)^{\top}$, $(0,1,-1)^{\top}$.
Therefore, the vector $u=(1,1,1)^{\top}\in\mathfrak{A}$ can be written as
\begin{equation*}
\left(\begin{array}{c}1 \\1 \\1 \\ \end{array}\right) = 
\left(\begin{array}{c}2 \\1 \\0 \end{array}\right)
-\left(\begin{array}{c}2 \\0 \\1 \end{array}\right)
+\left(\begin{array}{c}1\\0 \\2\end{array}\right)  \in \spanZ(\matrA_3).
\qed
\end{equation*}
\end{ex}

One immediate consequence of Lemma \ref{columnspan} is that the toric ideals in this section are not normal. This likely contributes to the difficulty of proving stabilization for chains induced by a non-square-free monomial.

\begin{cor}\label{not-normal}
Let ${\bf y}^\alpha$ be a non-square-free monomial in $\K[y_1,\ldots,y_k]$. For every $n\geq |\alpha|$, the toric ideal $I_n$ induced by the monomial ${\bf y}^\alpha$ is not normal.
\end{cor}

Recall from \cite[Proposition 13.5]{SturmfelsGBCP} that a toric ideal $I_\matrA$ is normal if and only if $pos(\matrA)\cap \spanZ(\matrA) = \spanN(\matrA)$, where $pos(\matrA)$ is the polyhedral cone defined by the columns of $\matrA$. 

\begin{proof}
Let $\alpha\in\N^k$ with $\gcd(\alpha)=1$, and let $\tau\in\symm_n$ be the cyclic permutation $\tau = (1\, 2 \cdots |\alpha|)$.  Realize $\alpha \in\Z^n$ by $\alpha =(\alpha_1,\ldots, \alpha_k,0,\ldots,0)^\top\in\Z^n$. Consider the following identity:
\begin{equation*}\label{eq:inthecone}
(1, \ldots,1,0,\ldots,0)^{\top} = \frac{1}{|\alpha|}(\alpha + \tau \alpha + \cdots + \tau^{|\alpha|-2}\alpha + \tau^{|\alpha|-1}\alpha).
\end{equation*}
By construction $z = (1, \ldots,1,0,\ldots,0)^{\top}\in pos(\matrA_n)$ and $|z| = |\alpha|$; thus, by Lemma \ref{columnspan} we see that $z \in pos(\matrA_n)\cap \spanZ(\matrA_n)$. However, since ${\bf y}^\alpha$ is not square-free, we must have $z \notin \spanN(\matrA_n)$.
\end{proof}

Although not required for the proof of Theorem \ref{mainthm2}, the Smith normal form (SNF) of the matrices $\matrA_n$ can be easily computed from Lemma \ref{columnspan}. For basic properties and algorithms involving the SNF over a principal ideal domain, we refer the reader to \cite{HGK2004, Yap2000}.

\begin{cor}\label{coroSNF}
Let $\alpha\in\N^k$ such that $\gcd(\alpha)=1$. For $n > k$ consider the matrix
$$ \matrA_n=\left( \sigma(\alpha_1,\ldots,\alpha_k,0,\ldots,0)^{\top} \in \Z^n : \sigma\in\symm_n \right).$$
The Smith normal form for $\matrA_n$ is $\diag(1,\ldots,1,|\alpha|)$.
\end{cor}
\begin{proof}
Use vectors $\widehat{h_j}$ from the proof of Lemma \ref{columnspan} to reduce the matrix $\matrA_n$ to SNF $\diag(1,\ldots,1,d)$, for some $d\in\N$.  From Lemma \ref{columnspan}, we know that $\spanZ(\matrA_n)$ is $\mathfrak A=\{\beta\in \Z^n : |\beta| \equiv 0 \mod |\alpha|\}$. Since $\Z^n/\mathfrak A$ is a finitely generated $\Z$-module,  the fundamental decomposition theorem for modules \cite[Theorem 7.8.2]{HGK2004} implies that
\begin{equation*}
\Z^n/\mathfrak A \cong \Z/d\Z.
\end{equation*}
On the other hand, $\mathfrak A$ is the kernel of the map $\Z^n \rightarrow \Z/|\alpha|\Z$ given by $\beta \mapsto |\beta| \mod |\alpha|$; therefore, 
\begin{equation*}
\Z^n/\mathfrak A \cong \Z/|\alpha|\Z,
\end{equation*}
as $\Z$-modules. 
Hence, $\Z/|\alpha|\Z \cong \Z/d\Z$, which implies $d=|\alpha|$.
\end{proof}

Let $d=\deg f=|\alpha|$ and set $r=\max\{\alpha_1,\ldots, \alpha_k\}$. Consider now 
\begin{equation}\label{Bneqn}
\matrB_n:=\{(a_1,\ldots,a_n)^{\top} \in\Z^n : a_1+\cdots+a_n=d,\, 0\leq a_1,\ldots,a_n\leq r\}.
\end{equation}
There is a natural bijection between elements of $\matrB_n$ and multisubsets of $[n]$ of cardinality $d$ with at most $r$ repetitions. Let $\Gamma_n$ be the set of such multisubsets. Every $\a=(\a_1,\ldots,\a_n)\in\matrB_n$ is in bijection with $\widetilde{\a}\in\Gamma_n$ via:
\begin{equation}\label{eq:multisubsets}
\a=(\a_1,\ldots,\a_n)\longleftrightarrow \widetilde{\a}=\{1^{\a_1},2^{\a_2},\ldots,n^{\a_n}\}.
\end{equation}

Let $\rtilde_n :=\K\left[X_{\Gamma_n}\right]$.  When $\matrB_n$ is viewed as a matrix with rows indexed by $t_i$ (for $i\in[n]$) and columns indexed by $x_{\tilde{\a}}$ (for $\widetilde{\a}\in\Gamma_n$), it defines a semigroup homomorphism that lifts to a homomorphism of $\K$-algebras:
$$\widetilde{\phi}_n:\rtilde_n\longrightarrow T_n.$$
By definition, $\matrA_n\subseteq \matrB_n$, and this inclusion gives an embedding $\eta: R_n \hookrightarrow \rtilde_n$. Also, $\widetilde{\phi}_n$ extends the map $\phi_n$ in the sense that $\phi_n =\widetilde{\phi}_n\circ \eta$. Therefore, we have the following commutative diagram:
$$
\xymatrix{
R_n \ar@{^{(}->}[r]^\eta \ar[d]_{\phi_n} & \widetilde{R}_n \ar[ld]^{\widetilde{\phi}_n}\\
 T_n
}
$$

\begin{ex}\label{example Bn}
Let $n=3$ and $\alpha=(2,1)$. Then $\rtilde_3=\K[x_{123}, x_{112}, x_{113}, x_{122}, x_{223}, x_{133}, x_{233}]$, and the following table
represents the matrix $\matrB_3$ associated to the homomorphism $\widetilde{\phi}_3$:
\begin{center}
\begin{tabular}{cccccccc}
 &$x_{123}$ & $x_{112}$ & $x_{113}$ & $x_{122}$ & $x_{223}$ & $x_{133}$ & $x_{233}$ \\
$t_1$ & 1 & 2 &2 &1& 0& 1& 0 \\
$t_2$ & 1 & 1 &0 &2 &2 &0 &1 \\
$t_3$ & 1 & 0 &1 &0 &1 &2 &2
\end{tabular}
\end{center}
\qed
\end{ex}

We next note the following fact, easily derived using \cite[Theorem 14.2]{SturmfelsGBCP}, as it provides a quadratic reduced Gr\"obner basis for $\widetilde{I}_n$. These generators can be obtain from the quadratic generators of any Gr\"obner basis for $\widetilde{I}_n$.

\begin{lem}\label{quadgeneration}
The ideal $\widetilde{I}_n\subseteq \rtilde_n$ is generated by the quadratic binomials of any Gr\"obner basis.
\end{lem}

In particular, since finite Gr\"obner bases always exist, $\widetilde{I}_n$ has a finite set of quadratic binomials generating it.

We now explain the key idea in our proof of Theorem \ref{mainthm2}.  Since the map $\widetilde{\phi}_n$ extends $\phi_n$, we have $I_n\hookrightarrow \widetilde{I}_n$.  Suppose that $\widetilde{I}_n = \< \widetilde{G}_n\>$ for some set $\widetilde{G}_n \subseteq \rtilde_n$, and that we can find a $\K$-algebra homomorphism $\mu$ making the following diagram commutative: 
\begin{equation}\label{maindiagram}
\xymatrix{
R_n \ar@{^{(}->}[rr] \ar@{^{(}->}[rd]^\eta \ar[rddd]_{\phi_n} && R_n^{\pm} \ar[lddd]^{\psi_n} 
\\ & \widetilde{R}_n \ar@{-->}[ru]^\mu \ar[dd]|{\widetilde{\phi}_n} \\ &\\
& T_n^{\pm}
}
\end{equation}
Then, as is easily checked, $\mu(\widetilde{G}_n)$ will be a generating set for $I^{\pm}_n$.  If, in addition, the $\widetilde{G}_n$ can themselves be finitely generated up to symmetry and $\mu$ is equivariant\footnote{The term \textit{equivariant} for the map $\mu$ signifies that $\mu(\sigma h) = \sigma \mu(h)$ for any $\sigma \in \symm_n$ and $h \in \widetilde{R}_n$.}, then we have generated the whole Laurent chain $I_{\circ}^{\pm}$ up to the symmetric group.  As the proof of the following proposition explains, the existence of such a $\mu$ is guaranteed by Lemma \ref{columnspan}.

\begin{prop}\label{mu_construction}
Fix $\alpha=(\alpha_1,\ldots,\alpha_k) \in\N^n$ and let $f={\bf y}^\alpha \neq 1$. For each $n > k$, there exists an equivariant $\K$-algebra homomorphism $\mu:\rtilde_n\rightarrow R_n^{\pm}$ that makes the diagram \eqref{maindiagram} commute. 
\end{prop}
\begin{proof}
Consider a multisubset $\widetilde{\a}\in\Gamma_n$.  If $x_{\tilde{\a}}\in \eta(R_n)\subseteq \rtilde_n$, then define $\mu(x_{\tilde{\a}}) :=\eta^{-1}(x_{\tilde{\a}})$. 
Assume $x_{\tilde{\a}}\notin \eta(R_n)$. Since $\widetilde{a}$ is in bijection with $a\in\mathcal B_n$ as in \eqref{eq:multisubsets}, we have $|a| = |\alpha|$. By Lemma \ref{columnspan}, we can find integers $B = \{b_1,\ldots,b_M\} \subset \Z$ such that 
$$
a=\sum_{i=1}^M b_i u_i,
$$
with $M={n \choose k}k!$ and $u_i \in \matrA_n$. Let $B^+=\{b_i\in B : b_i\in \Z_{>0}\}$ and $B_-=\{b_i\in B : b_i\in\Z_{<0}\}$.  Consider the fraction 
\begin{equation}\label{fraction}
\mathfrak q:= \frac{\prod_{b_i \in B^+} x_{u_i}^{b_i}}{\prod_{b_i \in B_{-}} x_{u_i}^{-b_i}}.
\end{equation}
Clearly $\mathfrak q \in R_n^{\pm}$, so we can define $\mu(x_{\tilde{\a}}) := \mathfrak q \in R_n^{\pm}$.

Extend $\mu$ to $\rtilde_n$ by linearity. 
By construction, $\mu$ makes the diagram (\ref{maindiagram}) commute since for $\atilde\in\Gamma_n$, one can verify that $\psi_n(\mu(x_{\atilde}))=\widetilde{\phi}_n(x_{\atilde})=\prod_{i=1}^n t_i^{\a_i}$.
\end{proof}

\begin{rem}
The above construction of $\mu$ is not necessarily unique as it depends on the representation of $\mathfrak q$.
\end{rem}

\begin{ex}\label{ex fraction}
Continuing from Example \ref{integer combination}, we want to map $x_{123}\in\rtilde_3$ to a fraction in $R_3^{\pm}$ that only involves the indeterminates of $\rtilde_3$ corresponding to those of $R_3$:
$$
\mu(x_{123})=\frac{x_{112}x_{331}}{x_{113}}.
$$
We also want this fraction to have the same image under $\psi_3$ as $x_{123}$ has under $\widetilde{\phi}_3$.  Indeed, we have $\widetilde{\phi}_3(x_{123})=t_1t_2t_3$ and 
\begin{eqnarray*}
\psi_3(\mu(x_{123})) = \frac{\psi_3(x_{112}x_{331})}{\psi_3(x_{113})} = \frac{\phi_3(x_{12})\phi_3(x_{31})}{\phi_3(x_{13})} = \frac{(t_1^2t_2)(t_3^2t_1)}{(t_1^2t_3)} = t_1t_2t_3.
\end{eqnarray*}
\qed
\end{ex}

We are finally in position to prove Theorem \ref{mainthm2}.

\begin{proof}[Proof of Theorem \ref{mainthm2}]
Let $\widetilde{I}_n=\ker\widetilde{\phi}_n$; this ideal is generated by binomials of the form
$$
x_{\tilde a}x_{\tilde b}\cdots x_{\tilde c}-x_{\tilde a'}x_{\tilde b'}\cdots x_{\tilde c'},
$$
in which $\atilde\cup\btilde\cdots\cup{\gtilde}=\atilde'\cup\btilde'\cdots\cup{\gtilde'}$ as a union of multisets \cite[Remark 14.1]{SturmfelsGBCP}. 
From Lemma \ref{quadgeneration}, there is a finite generating set $\mathcal{G}_n$ of $\widetilde{I}_n$ consisting of quadratic binomials. Let $G_n$ be a finite set of generators for $I_n$. Note that $\eta(I_n)\subseteq \widetilde{I}_n$ and so $\eta(G_n)\subseteq \widetilde{I}_n$.  For $g\in G_n$, we can write 
\begin{equation}\label{imageofeta}
\eta(g)=\sum_{\tilde{p}\in\mathcal{G}_n} h_{\tilde{p}} \widetilde{p}, \hspace{5pt}\text{ with }h_{\tilde{p}}\in\rtilde_n.
\end{equation}
We know $G_n$ is a generating set for $I_n^{\pm}$, but we give another generating set for $I_n^\pm$ in terms of $\mathcal{G}_n$. 

Applying the map $\mu$ from Proposition \ref{mu_construction} to both sides of expression \eqref{imageofeta}, we have
$$
g = \mu(\eta(g)) = \sum_{\tilde{p}\in\mathcal{G}_n} \mu(h_{\tilde{p}}) \mu(\widetilde{p}).
$$
Moreover, $\mu(\widetilde{p})\in I_n^\pm$.  It follows that $I_n^\pm=\langle \mu(\widetilde{p}) : \widetilde{p}\in \mathcal{G}_n \rangle_{R_n^{\pm}}$. Since $\widetilde{p}\in \mathcal{G}_n$ is a quadratic binomial, 
\[
\widetilde{p} = x_{\tilde a}x_{\tilde b} - x_{\tilde a'}x_{\tilde b'}, \ \  \text{with $\atilde\cup \btilde = \atilde'\cup \btilde'$ as multisets}.
\] 
The cardinality of each of $\atilde,\, \btilde$ is $d=|\alpha|$, and so the number of distinct numbers in $\atilde\cup \btilde$ is at most $2d$.
In particular, $\mu(\widetilde{p})\in   \<\symm_n I_{2d}^\pm\>_{R_n^\pm}$ for $n\geq 2d$. 
Thus, $I_\circ^\pm$ stabilizes with bound $N=2d$. 
\end{proof}

\begin{ex}
Continuing with Example \ref{ex:2ndThm}, let $g=x_{3 9}x_{7 9} - x_{37} x_{97} \in I_9$. Under the inclusion $\eta : R_6\hookrightarrow \widetilde R_6$, we have $\eta (g) = x_{339}x_{779} - x_{337}x_{799}$. From Lemma \ref{quadgeneration}, the ideal $\widetilde I_9$ is generated by quadratic binomials. We can write $\eta(g)$ in terms of those generators; in this case,
$$
\eta(g) = (x_{339}x_{779} - x_{379}^2) - (x_{337}x_{799} - x_{379}^2)\in \widetilde I_9.
$$
Let $\widetilde p_1 :=x_{337}x_{779} - x_{379}^2$ and $\widetilde p_2:=x_{337}x_{799} - x_{379}^2$. We have $\widetilde p_1 = \sigma \widetilde q_1$ and $\widetilde p_2 = \sigma \widetilde q_2$ for the following $\widetilde q_1, \widetilde q_2 \in \widetilde{I}_6$ (actually $\widetilde{I}_3$ in this case) and $\sigma = (1\, 3\, 9)(2\, 7) \in \symm_9$:
\[ \widetilde q_1 = x_{113}x_{223} - x_{123}^2, \ \   \widetilde q_2 = x_{112}x_{233} - x_{123}^2.\]
Thus, $\mu(\widetilde p_1) =  \sigma \mu(q_1)$ and   $\mu(\widetilde p_2) =  \sigma \mu(q_2)$ since $\mu$ is equivariant 
 and  $\mu(q_1)$ and $ \mu(q_2)$ generate $g$ up to symmetry:
\[ g=   \sigma \left(x_{12}x_{23}-\frac{x_{12}^2x_{31}^2}{x_{13}^2}\right) - \sigma \left(x_{12}x_{32}-\frac{x_{12}^2x_{31}^2}{x_{13}^2}\right).\]
\qed
\end{ex}

\subsection{Toric ideals induced by $y_1^2 y_2$}\label{sec:themainexample}

Theorem \ref{mainthm2} provides evidence that chains of ideals induced by monomials stabilize. The simplest (unknown) case is when $ f=y_1^2 y_2$. Here, we present an explicit computation of the generators for the corresponding Laurent chain that is different from Algorithm \ref{algorithmourthm}.  We hope to illustrate some of the complexity of the general problem and also to elaborate on other approaches for tackling Conjecture \ref{conjwmonomial}.

For $n\geq 2$, let $I_n$ be the toric ideal induced by the monomial $y_1^2y_2$. Let $\matrA_n\in\Z^{n\times {n\choose k}k!}$ be the matrix that defines the semigroup homomorphism $\phi_n$ such that $I_n=\ker \phi_n$ (recall Definition \ref{ker_ind_poly}). 
For example, when $n=5$ we have
 $$
 \matrA_5=\left(
{\small 
\begin{array}{cccccccccccccccccccc}
 2 & 2 & 2 & 2 & 1 & 1 & 1 & 1 & 0 & 0 & 0 & 0 & 0 & 0 & 0 & 0 & 0 & 0 & 0 & 0 \\
 1 & 0 & 0 & 0 & 2 & 0 & 0 & 0 & 2 & 2 & 2 & 1 & 1 & 1 & 0 & 0 & 0 & 0 & 0 & 0 \\
 0 & 1 & 0 & 0 & 0 & 2 & 0 & 0 & 1 & 0 & 0 & 2 & 0 & 0 & 2 & 2 & 1 & 1 & 0 & 0 \\
 0 & 0 & 1 & 0 & 0 & 0 & 2 & 0 & 0 & 1 & 0 & 0 & 2 & 0 & 1 & 0 & 2 & 0 & 2 & 1 \\
 0 & 0 & 0 & 1 & 0 & 0 & 0 & 2 & 0 & 0 & 1 & 0 & 0 & 2 & 0 & 1 & 0 & 2 & 1 & 2
\end{array}
}
 \right).
 $$

When the columns of $\matrA_n$ are ordered lexicographically, a basis for $\ker_{\Z}(\matrA_n)$ as a $\Z$-module can be described as follows:
\begin{equation}\label{matrixBlock}
\ker_{\Z}(\matrA_n) = \left(\begin{array}{c} \mathcal{K}_n\\ \mathcal{I}_{d-n} \end{array}\right),
\end{equation}
where $\mathcal{I}_{d-n}$ is the $(d-n)\times(d-n)$ identity matrix and $\mathcal{K}_n$ is a matrix whose structure we now describe.  Let $c_r\in\Z^{n-2}$ be the row vector whose entries are all equal to $r$. Then,
$$
\mathcal{K}_n = \left( \begin{array}{cccc} \mathcal{L}_1 & \mathcal{L}_2 & \mathcal{L}_3 & \mathcal{L}_4
\end{array}
 \right),
$$
in which

\[
-\mathcal{L}_1 = \left( \begin{array}{c} 
c_{-2} \\
2\cdot \mathcal{I}_{n-2} \\
c_{1}
 \end{array}
 \right), 
\hspace{.2cm}
-\mathcal{L}_2 = \left( \begin{array}{c} 
c_{-2}\\
 \mathcal{I}_{n-2} \\
c_{2} 
 \end{array}
 \right),
\hspace{.2cm}
-\mathcal{L}_3 = \left( \begin{array}{c} 
c_{-3}\\
2\cdot \mathcal{I}_{n-2} \\
c_{2} 
 \end{array}
 \right),
 \hspace{.2cm}
-\mathcal{L}_4 = \left( \begin{array}{c} 
c_{-4}\\
\matrA_{n-2} \\
c_{2} 
 \end{array}
 \right).
 \] 
 For instance, when $n=5$, the integer kernel of $\matrA_5$ has the following $\Z$-basis
 $$
 \ker_{\Z}(\matrA_5) = \left(
{
\renewcommand{\arraystretch}{0.8}
\begin{array}{rrrrrrrrrrrrrrr}
 2 & 2 & 2 & 2 & 2 & 2 & 3 & 3 & 3 & 4 & 4 & 4 & 4 & 4 & 4 \\
 -2 & 0 & 0 & -1 & 0 & 0 & -2 & 0 & 0 & -2 & -2 & -1 & -1 & 0 & 0 \\
 0 & -2 & 0 & 0 & -1 & 0 & 0 & -2 & 0 & -1 & 0 & -2 & 0 & -2 & -1 \\
 0 & 0 & -2 & 0 & 0 & -1 & 0 & 0 & -2 & 0 & -1 & 0 & -2 & -1 & -2 \\
 -1 & -1 & -1 & -2 & -2 & -2 & -2 & -2 & -2 & -2 & -2 & -2 & -2 & -2 & -2 \\
 1 & 0 & 0 & 0 & 0 & 0 & 0 & 0 & 0 & 0 & 0 & 0 & 0 & 0 & 0 \\
 0 & 1 & 0 & 0 & 0 & 0 & 0 & 0 & 0 & 0 & 0 & 0 & 0 & 0 & 0 \\
 0 & 0 & 1 & 0 & 0 & 0 & 0 & 0 & 0 & 0 & 0 & 0 & 0 & 0 & 0 \\
 0 & 0 & 0 & 1 & 0 & 0 & 0 & 0 & 0 & 0 & 0 & 0 & 0 & 0 & 0 \\
 0 & 0 & 0 & 0 & 1 & 0 & 0 & 0 & 0 & 0 & 0 & 0 & 0 & 0 & 0 \\
 0 & 0 & 0 & 0 & 0 & 1 & 0 & 0 & 0 & 0 & 0 & 0 & 0 & 0 & 0 \\
 0 & 0 & 0 & 0 & 0 & 0 & 1 & 0 & 0 & 0 & 0 & 0 & 0 & 0 & 0 \\
 0 & 0 & 0 & 0 & 0 & 0 & 0 & 1 & 0 & 0 & 0 & 0 & 0 & 0 & 0 \\
 0 & 0 & 0 & 0 & 0 & 0 & 0 & 0 & 1 & 0 & 0 & 0 & 0 & 0 & 0 \\
 0 & 0 & 0 & 0 & 0 & 0 & 0 & 0 & 0 & 1 & 0 & 0 & 0 & 0 & 0 \\
 0 & 0 & 0 & 0 & 0 & 0 & 0 & 0 & 0 & 0 & 1 & 0 & 0 & 0 & 0 \\
 0 & 0 & 0 & 0 & 0 & 0 & 0 & 0 & 0 & 0 & 0 & 1 & 0 & 0 & 0 \\
 0 & 0 & 0 & 0 & 0 & 0 & 0 & 0 & 0 & 0 & 0 & 0 & 1 & 0 & 0 \\
 0 & 0 & 0 & 0 & 0 & 0 & 0 & 0 & 0 & 0 & 0 & 0 & 0 & 1 & 0 \\
 0 & 0 & 0 & 0 & 0 & 0 & 0 & 0 & 0 & 0 & 0 & 0 & 0 & 0 & 1
\end{array}
}
\right).
$$

For each $n$, the elements of $\ker_{\Z}(\matrA_n)$ are $\Z$-linear combinations of the columns of the matrix \eqref{matrixBlock}. For each $i=1,\ldots,4$, we can realize the columns of $\mathcal{L}_i$ as the first column of $\mathcal{L}_i$ after applying a permutation $\sigma\in\symm_n$ to it. For instance, when $n=5$, the first column of $\mathcal{L}_1$ is the vector $(2,-2,0,0,-1,1,0,\ldots,0)^{\top}\in\Z^{20}$, which corresponds to the binomial $x_{12}^2 x_{31} - x_{13}^2 x_{21}$. If we apply the transposition $(3\, 5)\in \symm_5$ to this element, we get $x_{12}^2 x_{51} - x_{15}^2 x_{21}$, whose corresponding integer vector is precisely the third column of $\mathcal{L}_1$; namely, $(2,0,0,-2,-1,0,0,1,0,\ldots,0)^{\top}\in\Z^{20}$.

In general, for every $n$ and for $i=1,2,3$, the transposition $(3 \, j)$ with $4\leq j \leq n$ applied to the binomial corresponding to the first column of $\mathcal{L}_i$ will be equal to the binomial whose support corresponds to the $(j{-}2)$-th column of $\mathcal{L}_i$. For $\mathcal{L}_4$, instead of transpositions, we use those permutations that send the pair $(3,4)$ to $(i,j)$ for $3\leq i\neq j \leq n$ to write those binomials corresponding to the columns of $\mathcal{L}_4$ in terms of the first column of $\mathcal{L}_4$.  For instance, the binomial $x_{12}^4 x_{34} - x_{14} x_{13}^2 x_{21}^2$ has support the first column of $\mathcal{L}_4$.  When we apply the permutation $(3 \ 4 \ 5)\in\symm_5$ to this binomial, we get $x_{12}^4 x_{45} - x_{15} x_{14}^2 x_{21}^2$, which has support the 5th column of $\mathcal{L}_4$.  Consider the set
\[H^{\pm}  =\{ x_{12}^2 x_{31} - x_{13}^2 x_{21}, x_{12}^2 x_{23} - x_{13} x_{21}^2, x_{12}^3 x_{32} - x_{13}^2 x_{21}^2 , x_{12}^4 x_{34} - x_{13}^2 x_{14} x_{21}^2\}\] of binomials corresponding to the first column of each $\mathcal L_i$. The action of $\symm_5$ on $H^{\pm}$ produces generators for the Laurent ideal $I_5^\pm$ corresponding to the toric ideal $I_5$, by Lemma \ref{laurentLemma}.  In general for $n\geq 5$, the action of $\symm_n$ on $H^{\pm}$ produces generators for the Laurent ideal $I_n^\pm$. We thus obtain a generating set for the chain $I_{\circ}^\pm$ that depends only on the description of $\ker_{\Z}(\matrA_n)$ and is independent from the methods used in the proof of Theorem \ref{mainthm2}. 

Unfortunately, we could not generalize this technique to other cases as the combinatorics that describe $\ker_{\Z}(\matrA_n)$ in general becomes more complicated.  We also remark that the set $H^{\pm}$ fails to be a generating set for the (non-Laurent) chain of ideals induced by $y_1^2y_2$. 

\section{Algorithms}\label{algorithms}

The proof of Theorem \ref{mainthm2} suggests an algorithm to find the generators of a chain of Laurent toric ideals induced by a monomial ${\bf y}^\alpha$. We stated the existence of such an algorithm in Theorem \ref{ouralgthm} from the introduction. In this section we describe this algorithm and argue its correctness. A full implementation in \texttt{Macaulay2} \cite{M2} can be found in~\cite{WWW_ourcode}.

\begin{algorithm}[h!]
\caption{[Theorem \ref{ouralgthm}]}
\label{algorithmourthm}
\begin{algorithmic}[1]
\REQUIRE {Exponent vector $\alpha\in\N^k$}
\ENSURE {Generators for the chain of Laurent ideals defined by ${\bf y}^\alpha$ up to
symmetry}
\STATE $d := 2 |\alpha|$
\STATE Compute the matrix $\matrB_{d}$ (\ref{Bneqn})
\STATE\label{computegroebner} Compute the Gr\"{o}bner basis $\mathcal{G}$ of the toric ideal
$I_{\matrB_d}$
\FORALL {$g \in \mathcal{G}$}
   \FORALL {indeterminates $x_w$ in $g$}
       \IF {$x_w$ is not indexed by a permutation of $\alpha$}
           \STATE\label{createmu} $g =$ replace $x_w$ in $g$ by the monomial quotient $\mu(x_w)$
       \ENDIF
   \ENDFOR
\ENDFOR
\STATE\label{removeorbit}{ Remove redundant generators from $\mathcal{G}$}
\RETURN $\mathcal{G}$
\end{algorithmic}
\end{algorithm}

Given an exponent vector $\alpha\in \N^k$, the algorithm computes a set of generators for the chain of Laurent ideals defined by ${\bf y}^\alpha$ up to the action of the symmetric group. In the first steps, it considers all the integer partitions of $d =2 |\alpha|$ with parts at most $\max \alpha := \max \{\alpha_1,\ldots,\alpha_n\}$, and then constructs the matrix $\matrB_d$ by taking all the permutations of such partitions.

In step \ref{computegroebner}, the algorithm constructs the toric ideal $I_d$ that corresponds to the matrix $\matrB_d$ and computes its Gr\"{o}bner basis $\mathcal{G}$ (with respect to any term order). This Gr\"{o}bner basis computation is the most expensive step for large ideals. We decided to use the \texttt{Macaulay2} package \texttt{FourTiTwo}, which invokes one of the fastest routines, \texttt{4ti2}, specializing in computing Gr\"{o}bner bases for toric ideals \cite{4ti2}.

Step \ref{removeorbit} removes the redundant generators from $\mathcal{G}$. Using Lemma \ref{quadgeneration}, we start by removing all the non-quadratic generators from $\mathcal{G}$. We then remove the symmetric orbit of each of the remaining generators. To illustrate how drastically the number of generators is decreased after this step, consider once more the running example of Section \ref{stablaurchainmon}. When ${\bf y}^\alpha = y_1^2 y_2$, the Laurent toric chain has a stabilization bound at $n=6$; for this value of $n$, the toric ideal $I_6\subset R_6$ has $270$ minimal generators. When we lift to the ideal $\widetilde{I}_6 \subset \widetilde{R}_6$, we obtain $849$ minimal generators, but only $13$ modulo the action of the symmetric group. From those, we find that $11$ generate the corresponding Laurent ideal modulo the symmetric group. But after clearing denominators and common monomial factors, we found that only $8$ from those $11$ (exactly those $8$ that are presented in the introduction of Section \ref{stablaurchainmon}) form a generating set of the Laurent ideal $I_6^{\pm}$ modulo the action of the symmetric group. Since the number of generators increase when passing to the ring $\rtilde_n$, one way to improve speed on this orbit removal step is to remove the orbit after Step \ref{computegroebner} and again in Step \ref{removeorbit}.

The core of our algorithm is Step \ref{createmu}, where we turn Lemma \ref{columnspan} and Proposition \ref{mu_construction} into a computational tool. We unfold this step in Algorithm \ref{proc:constructmu} below.  This algorithm expresses every element of the column span of $\matrB_d$ as a linear combination of the column span of the matrix $\matrA_d$ using the construction found in the proof of Lemma \ref{columnspan}, and detailed in Algorithm \ref{proc:integerdecomposition} below. This integer decomposition is then used to create the map $\mu$ in (\ref{com_diagram}).

\begin{algorithm}[h!]
\floatname{algorithm}{Algorithm}
\caption{Construction of the map $\mu$}
\label{proc:constructmu}
\begin{algorithmic}[1]
\REQUIRE {indeterminate $x_w$, exponent vector $\alpha\in\N^k$ with $\gcd(\alpha)=1$}
\ENSURE {monomial quotient $\mu(x_w)$}
\STATE {Set $V:=\{\sigma \alpha :  \sigma\in\symm_n\}$}
\IF {$w\notin V$}
	\STATE {Write $w=b_1v_1+\cdots+b_rv_r$ with $b_i\in \Z$ and $v_i\in V$ for all $i\in[r]$ (Lemma \ref{columnspan})}
	\STATE {$indexB_+:=\{i\in[r] :  b_i>0\}$ and $indexB_-:=\{i\in[r] :  b_i<0\}$}
	\STATE {$numerator:=1$ and $denominator:=1$}
	\FORALL {$i\in indexB_+$}
		\STATE {$numerator = numerator \cdot x_{v_i}^{b_i}$ }
	\ENDFOR
	\FORALL {$i\in indexB_-$}
		\STATE{ $denominator=denominator\cdot x_{v_i}^{-b_i}$}
	\ENDFOR
	\RETURN {$numerator/denominator$}
	\ELSE
		\RETURN {$x_w$}
\ENDIF
\end{algorithmic}
\end{algorithm}

\begin{algorithm}[h!]
\floatname{algorithm}{Algorithm}
\caption{Integer decomposition of $\beta$ in terms of $\alpha$}
\label{proc:integerdecomposition}
\begin{algorithmic}[1]
\REQUIRE {Integer vector $\beta\in\Z^n$, exponent vector $\alpha\in\N^k$ with $\gcd(\alpha)=1$ and $|\alpha|$ dividing $|\beta|$}
\ENSURE {List $\{\{a_\sigma,v_\sigma\}: \sigma\in\symm_n\}$ such that $\beta=\sum_{\sigma\in\symm_n} a_\sigma v_\sigma$}, where $v_\sigma = \sigma\cdot \alpha$, and $a_\sigma\in \Z$
\STATE {Write $1=b_1\alpha_1+\cdots +b_k\alpha_k$, where $\alpha=(\alpha_1,\ldots, \alpha_k)$ and $b_i\in\Z$ }
\STATE {$q:=|\beta|/|\alpha|$}
\STATE {$L:=\{\{q, v_\alpha\}\}$, where $v_\alpha = (\alpha_1,\ldots,\alpha_k,0,\ldots,0)\in\Z^n$}
\FOR{$j$ from 1 to $n-1$}
	\IF {$\beta_j-\alpha_j\cdot q \neq0$}
		\FOR{$i$ from 1 to $k$}
			\STATE{$L=L\cup\{ \{ (\beta_j-\alpha_j\cdot q)b_i, v_{\sigma_{ij}} \}, \{-b_i(\beta_j-\alpha_j\cdot q), v_{\tau_j\sigma_{ij}}\} \}$, where $\sigma_{ij}$ and $\tau_j$ are as in the proof of Lemma \ref{columnspan}}
		\ENDFOR
	\ENDIF
\ENDFOR
\RETURN {$L$}
\end{algorithmic}
\end{algorithm}

We remark that in the union computation in Step 7 of Algorithm \ref{proc:integerdecomposition}, we add coefficients of matching pairs as in $\{a_\sigma,v_\sigma\} \cup \{b_\sigma,v_\sigma\} = \{a_\sigma+b_\sigma,v_\sigma\}$.

\section{Open problems and conjectures}\label{openprobconj}

Stabilization of chains of ideals is unexpected and important for applications. However, the problem of deciding whether a chain is stable under the action of a group seems difficult, even for the special case of the symmetric group. In this section, we present some conjectures based on computational evidence. We focus first when the ideals comprising the chain are toric ideals as they tend to have rich combinatorial structure; later, we turn to a more general setting and close with some problems that develop this topic further.

Motivated by the study of bounds on the Castelnuovo-Mumford regularity in algebraic geometry, Bayer and Mumford introduced in \cite{BM93} the \emph{degree-complexity} of a homogeneous ideal $I$ with respect to a term order $\preceq$ as the maximal degree in a reduced Gr\"obner basis of $I$, and this is the largest degree of a minimal generator of $\text{in}_{\preceq}(I)$. In our context, degree-complexity is important because it is closely related to stabilization of chains of ideals. For instance, in the proof of Theorem \ref{mainthm2}, we exploited the fact that the ideal $\widetilde{I}_n$ is binomial and has degree-complexity 2 for every $n$. On the other hand, if the ideals comprising a chain induced by a monomial do not have a degree-complexity bound, then stabilization is unlikely.


We pose the following problems based on our observations in Table \ref{degreechart} (computed using our software~\cite{WWW_ourcode}).

\begin{conj}
Let $\alpha=(\alpha_1,\alpha_2)$ with $\gcd(\alpha_1,\alpha_2)=1$ (suppose $\alpha_1\geq \alpha_2$). The degree-complexity of $I_n$ is of the form $2 \alpha_1-\alpha_2$ for all (but possible finitely many) ideals $I_n$ in the chain of (non-Laurent) toric ideals induced by the monomial ${\bf y}^\alpha$. 
\end{conj}

\begin{prob}\label{bound-degree-complexity}
Let $\alpha\in\N^k$ with $\gcd(\alpha)=1$; is the degree-complexity of $I_n$ constant as $n\rightarrow\infty$?
\end{prob}

\begin{table}[h!]
\label{degreechart}
\begin{center}
\begin{tabular}{|c||c|c|c|c|c|c|c|c|}
\hline
$\alpha\setminus n$ & 3&4&5&6&7&8\\
\hline\hline
$(1,1)$ & $1$ & 2 & 2 & 2 & 2 & 2 \\
\hline
$(2,1)$ & 3 & 3 & 3 & 3 & 3 & 3 \\
\hline
$(3,1)$ & 5 & 5 & 5 & 5 & 5 & 5\\
\hline
$(4,1)$ & 7 & 7 & 7 & 7 & 7 & 7\\
\hline
$(5,1)$ & 9 & 9 & 9 & 9 & 9 & 9\\
\hline
$(6,1)$ & 11 & 11 & 11 & 11 & 11 & - \\
\hline
$(7,1)$ & 13 & 13 & 13 & 13 & 13 & -\\
\hline
$(8,1)$ & 15 & 15 & 15 & 15 & - & - \\
\hline
$(3,2)$ &5 & 5 & 5 & 5 & 5 & 5 \\
\hline
$(4,2)$ & 3 & 3 & 3 & 3 & 3 & 3 \\
\hline
$(5,2)$ & 8 & 8 & 8 & 8 & 8 & 8 \\
\hline
$(6,2)$ & 5 & 5 & 5 & 5 & 5 & 5 \\
\hline
$(7,2)$ & 12 & 12 & 12 & 12 & 12 & - \\
\hline
$(1,3,2)$ & 3 & 3 &3 & 3 & - & - \\
\hline
$(4,3,2)$ & 3 & 5 & 5 & 5 & - & - \\
\hline
\end{tabular}
\caption{Degree-complexity of the toric ideal $I_n$ defined by ${\bf y}^\alpha$}
\end{center}
\end{table}

Recall the set $\mathcal B_n$ from \eqref{Bneqn}. Using the fact that $I_{\mathcal B_n}$ is generated by quadratics we show in this paper that for $\matrA_n\subseteq \mathcal B_n$, the chain of ideals $I_{\mathcal A_n}$ has a corresponding Laurent chain $I_{\mathcal A_n}^\pm$ that is stable under the action of $\symm_\P$. 
On the other hand, Conjecture~\ref{conjwmonomial} makes the stronger claim that the chain $I_{\matrA_n}$ stabilizes. While it is difficult to find subsets $\mathcal{C}_n \subseteq \mathcal B_n$ for which the chain of ideals $I_{\mathcal{C}_n}$ is stable under the action of $\symm_\P$, one might get some indications by solving the following problem.
\begin{prob}\label{prob:degreecomplxty}
Find combinatorially defined subsets $\mathcal{C}_n\subseteq \mathcal B_n$ such that the toric ideal $I_{\mathcal{C}_n}$ has constant degree-complexity as $n$ grows.
\end{prob}

This problem is of particular interest in algebraic statistics. For instance, in \cite[Conjecture 7.3]{HMCY2011}, it is conjectured that for any $T\geq 3$ and a fixed $S\geq 3$, the toric ideals of the homogeneous Markov chain model on $S$ states are generated by polynomials of degree at most $S-1$. This is an instance of Problem \ref{prob:degreecomplxty}, as for each $T\geq 3$, the design matrix of such a model is precisely a subset of the matrix $\matrB_n$ for $n=T$.
 
The early stabilization of structured chains appears to be common.  It would be interesting to construct examples of chains with nontrivial lower bounds on stabilization.
\begin{prob}
Let $f(d)$ be an increasing function $f: \mathbb \P \to \mathbb \P$.  Find a family of invariant chains $\left\{I_{\circ}^{(d)} \right\}_{d=1}^{\infty}$ (over $\mathcal{R}_{\P}$ or $R_{\P}$) which have stabilization bound at least $f(d)$. 
\end{prob}
More specifically, we ask whether a linear lower bound holds for the chains in Theorem \ref{mainthm2} and their Laurent counterparts.
\begin{prob}
Is there a constant $C > 0$ such that the chains $\left\{I_{\circ}^{(d)} \right\}_{d=1}^{\infty}$ from Theorem \ref{mainthm2} must have stabilization bounds at least $f(d) = Cd$.  
\end{prob}
\section{Acknowledgements}\label{sec:ack}
We thank Frank Sottile and Seth Sullivant for useful comments.  We also thank Kelli Talaska for helping with the proof of Lemma \ref{columnspan} and Matthias Aschenbrenner for ideas concerning nice orderings.


\bibliographystyle{amsplain}

\bibliography{HM-InvarChainStab}

\providecommand{\bysame}{\leavevmode\hbox to3em{\hrulefill}\thinspace}
\providecommand{\MR}{\relax\ifhmode\unskip\space\fi MR }
\providecommand{\MRhref}[2]{%
  \href{http://www.ams.org/mathscinet-getitem?mr=#1}{#2}
}
\providecommand{\href}[2]{#2}
\begin{thebibliography}{10}

\bibitem{4ti2}
4ti2 team, \emph{4ti2---a software package for algebraic, geometric and
  combinatorial problems on linear spaces}, {A}vailable at www.4ti2.de.

\bibitem{Ahlbrandt-Ziegler}
G~Ahlbrandt and M~Ziegler, \emph{Quasi-finitely axiomatizable totally
  categorical theories, stability in model theory}, Pure Appl. Logic (1984),
  no.~30, 63--82.

\bibitem{AHT10}
Satochi {Aoki}, Hisayuki {Hara}, and Akimichi {Takemura}, \emph{{Minimal and
  minimal invariant Markov bases of decomposable models for contingency
  tables}}, Bernoulli \textbf{16} (2010), no.~1, 208--233.

\bibitem{AH07}
Matthias Aschenbrenner and Christopher~J. Hillar, \emph{Finite generation of
  symmetric ideals}, Trans. Amer. Math. Soc. \textbf{359} (2007), 5171--5192.

\bibitem{BM93}
Dave Bayer and David Mumford, \emph{What can be computed in algebraic
  geometry?}, Computational algebraic geometry and commutative algebra
  ({C}ortona, 1991), Sympos. Math., XXXIV, Cambridge Univ. Press, Cambridge,
  1993, pp.~1--48. \MR{1253986 (95d:13032)}

\bibitem{Brouwer-Draisma11}
Andries~E. Brouwer and Jan Draisma, \emph{Equivariant {G}r\"obner bases and the
  {G}aussian two-factor model}, Math. Comp. \textbf{80} (2011), no.~274,
  1123--1133. \MR{2772115}

\bibitem{Camina-Evans}
Alan~R. Camina and David~M. Evans, \emph{Some infinite permutation modules},
  Quart. J. Math. Oxford \textbf{42} (1991), no.~2, 15--26.

\bibitem{Cohen67}
D.~E. Cohen, \emph{On the laws of a metabelian variety}, Journal of Algebra
  \textbf{5} (1967), no.~3, 267--273.

\bibitem{Cohen77}
Daniel~A. Cohen, \emph{Closure relations, buchberger's algorithm, and
  polynomials in infinitely many variables}, pp.~78--87, Springer-Verlag,
  London, UK, 1987.

\bibitem{CLO07}
David~A. Cox, John~B. Little, and Donald O'Shea, \emph{Ideals, varieties, and
  algorithms: an introduction to computational algebraic geometry and
  commutative algebra}, 3rd. ed., Undergraduate texts in mathematics, vol.~10,
  Springer, New York-Berlin, 2007.

\bibitem{DEMO09}
Zajj Daugherty, Alexander~K. Eustis, Gregory Minton, and Michael~E. Orrison,
  \emph{Voting, the symmetric group, and representation theory}, The American
  Mathematical Monthly \textbf{116} (2009), no.~8, pp. 667--687 (English).

\bibitem{Draisma2010}
Jan Draisma, \emph{Finiteness for the k-factor model and chirality varieties},
  Adv. Math. \textbf{223} (2010), 243--256.

\bibitem{DraismaKuttler11}
Jan Draisma and Jochen Kuttler, \emph{Bounded-rank tensors are defined in
  bounded degree}, arXiv:1103.5336 (2011).

\bibitem{DSS07}
Mathias Drton, Bernd Sturmfels, and Seth Sullivant, \emph{Algebraic factor
  analysis: tetrads, pentads and beyond}, Probab. Theory Relat. Fields
  \textbf{138} (2007), 463--493.

\bibitem{Eisenbud95}
David Eisenbud, \emph{Commutative algebra with a view toward algebraic
  geometry}, Graduate texts in mathematics, vol. 150, Springer-Verlag, New
  York, 1995.

\bibitem{EisSturm96}
David Eisenbud and Bernd Sturmfels, \emph{Binomial ideals}, Duke Math. J.
  \textbf{84} (1996), no.~1, 1--45. \MR{1394747 (97d:13031)}

\bibitem{GMV10}
J.~I. {Garc{\'{\i}}a-Garc{\'{\i}}a}, M.~A. {Moreno-Fr{\'{\i}}as}, and
  A.~{Vigneron-Tenorio}, \emph{{On the decomposable semigroups and their
  applications in Algebraic Statistics}}, arXiv:1006.2557 (2010).

\bibitem{M2}
Daniel~R. Grayson and Michael~E. Stillman, \emph{Macaulay2, a software system
  for research in algebraic geometry}, Available at
  http://www.math.uiuc.edu/Macaulay2/.

\bibitem{TakemuraHara10b}
Hisayuki Hara and Akimichi Takemura, \emph{A markov basis for two-state toric
  homogeneous markov chain model without initial paramaters}, arXiv:1005.1717
  (2010).

\bibitem{TakemuraHara10a}
\bysame, \emph{Markov chain monte carlo test of toric homogeneous markov
  chains}, arXiv:1004.3599 (2010).

\bibitem{HMCY2011}
David {Haws}, Abraham {Mart\'in del Campo}, and Ruriko {Yoshida}, \emph{{Degree
  bounds for a minimal Markov basis for the three-state toric homogeneous
  Markov chain model}}, arXiv:1108.0481 (2011).

\bibitem{HGK2004}
Michiel Hazewinkel, Nadezhda~Mikha\u{\i}lovna Gubareni, and Vladimir~V.
  Kirichenko, \emph{Algebras, rings and modules}, Mathematics and its
  applications, vol. 575, Kluwer Academic Publishers, Dordrecht London, 2004.

\bibitem{Hig52}
G.~Higman, \emph{Ordering by divisibility in abstract algebras}, Proc. London
  Math. Soc. \textbf{2} (1952), 326--336.

\bibitem{WWW_ourcode}
Christopher~J. Hillar and Abraham Mart{\'{i}}n~del Campo, \emph{Code to compute
  symmetric invariant generators}, 2010, \url
  {http://www.math.tamu.edu/~asanchez/Files/Code/symmChainGens.m2}.

\bibitem{HS}
Christopher~J. Hillar and Seth Sullivant, \emph{Finite gr\"oebner bases in
  infinite dimensional polynomial rings and applications}, Adv. Math. (2011).

\bibitem{Hosten2007}
Serkan Ho\c{s}ten and Seth Sullivant, \emph{A finiteness theorem for markov
  bases of hierarchical models}, J. Comb. Theory Ser. A \textbf{114} (2007),
  no.~2, 311--321.

\bibitem{Kruskal72}
Joseph~B. Kruskal, \emph{The theory of well-quasi-ordering: {A} frequently
  discovered concept}, J. Combinatorial Theory Ser. A \textbf{13} (1972),
  297--305. \MR{0306057 (46 \#5184)}

\bibitem{Kuo06}
Eric~H. Kuo, \emph{Viterbi sequences and polytopes.}, Journal of Symbolic
  Computation (2006), 151--163.

\bibitem{LM2004}
J.~M. Landsberg and L.~Manivel, \emph{On the ideals of secant varieties of
  segre varieties}, Found. Comput. Math. \textbf{4} (2004), 397--422.

\bibitem{CCA}
Ezra Miller and Bernd Sturmfels, \emph{Combinatorial commutative algebra},
  Graduate text in mathematics, vol. 227, Springer-Verlag, New York, 2005.

\bibitem{NW1}
C.~St. J.~A. Nash-Williams, \emph{On well-quasi-ordering finite trees}, Proc.
  Cambridge Philos. Soc. (1963), no.~59, 833--835.

\bibitem{Raicu2010}
Claudiu {Raicu}, \emph{{The GSS Conjecture}}, arXiv:1011.5867 (2010).

\bibitem{RSU67}
Ernst Ruch, Alfred Sch\"{o}nhofer, and Ivar Ugi, \emph{Die vandermondesche
  determinante als n\"{a}herungsansatz f\"{u}r eine chiralit\"{a}tsbeobachtung,
  ihre verwendung in der stereochemie und zur berechnung der optischen
  aktivit\"{a}t}, Theoretical Chemistry Accounts: Theory, Computation, and
  Modeling (Theoretica Chimica Acta) \textbf{7} (1967), 420--432,
  10.1007/BF00526408.

\bibitem{Santos2003}
Francisco Santos and Bernd Sturmfels, \emph{Higher lawrence configurations}, J.
  Comb. Theory Ser. A \textbf{103} (2003), no.~1, 151--164.

\bibitem{DLS09}
Roberto~La Scala and Viktor Levandovskyy, \emph{Letterplace ideals and
  non-commutative gr{\"o}bner bases}, Journal of Symbolic Computation
  \textbf{44} (2009), no.~10, 1374--1393.

\bibitem{Snowden10}
Andrew {Snowden}, \emph{{Syzygies of Segre embeddings}}, arXiv:1006.5248
  (2010).

\bibitem{SturmfelsGBCP}
B.~Sturmfels, \emph{Gr\"obner bases and convex polytopes}, American
  Mathematical Society, Univ. Lectures Series, no.~8, Providence, Rhode Island,
  1996.

\bibitem{SturmfelsSullivant05}
Bernd Sturmfels and Seth Sullivant, \emph{Toric ideals of phylogenetic
  invariants}, Journal of Computational Biology \textbf{12} (2005), no.~2,
  204--228.

\bibitem{Yap2000}
Chee~Keng Yap, \emph{Fundamental problems of algorithmic algebra}, Oxford
  University Press, New York, 2000. \MR{1740761 (2000m:12014)}

\end{thebibliography}
\label{sec:biblio}

\end{document}